\documentclass[11pt]{article}

\usepackage{amsmath,amssymb,amsthm,amscd,eucal}
\usepackage{latexsym}
\usepackage{color}
\usepackage{enumerate}

\newtheorem{thm}[equation]{Theorem}
\newtheorem{cor}[equation]{Corollary}
\newtheorem{lemma}[equation]{Lemma}
\newtheorem{prop}[equation]{Proposition}

\newtheorem*{conj*}{Conjecture}
\theoremstyle{definition}
\newtheorem{defn}[equation]{Definition}

\newtheorem{remark}[equation]{Remark}
\newtheorem{exam}[equation]{Example}
\newtheorem{exams}[equation]{Examples}
\newtheorem{asspts}[equation]{Assumptions}
 
\numberwithin{equation}{section} 

\newcommand{\tor}{{\rm tor}}
\newcommand{\ann}{\mathsf{Ann}}

\newcommand{\ot}{\otimes}

\newcommand{\FF}{\mathbb{F}}  
\newcommand{\EE}{\mathsf{E}}  
\newcommand{\ZZ}{\mathbb{Z}}

\newcommand{\A}{\mathsf{A}} 
\newcommand{\BB}{\mathsf{B}}

\newcommand{\DD}{\mathsf{R}}

\newcommand{\LL}{\mathsf{L}} 
 
\newcommand{\M}{\mathsf{M}} 
\newcommand{\Lmg}{\LL(\m,\bar g)}
\newcommand{\Lmgb}{\LL(\m,z_\beta)}
\newcommand{\m}{\mathfrak{m}}
\newcommand{\N}{\mathsf{N}} 
\newcommand{\RR}{\mathsf{R}} 
\newcommand{\Sf}{\mathsf{S}}

\newcommand{\PR}{\mathsf{Q}}

\newcommand{\UU}{\mathsf{U}}  
\newcommand{\V}{\mathsf{V}}

\newcommand{\W}{\mathsf{W}}

\def \Max {\mathfrak{max}}  

\newcommand{\spann}{\mathsf{span}}

\newcommand\End{\mathsf {End}}

\newcommand\chara{\mathsf {char}}
\def\inder{\mathsf {Inder_\FF}}
\def\der{\mathsf {Der_\FF}}

\def\hoch{\mathsf{HH^1}}

\def\aut{\mathsf {Aut_\FF}}

\def\m{{\mathfrak m}}

\def\degg{\mathsf{deg\,}}
\def\cent1{\mathsf{Z}(\A_1)}
\def\centh{\mathsf{Z}(\A_h)}

\newcommand{\modd}{\mathsf{\, mod\,}}  

\usepackage[letterpaper,margin=1.48in]{geometry}

\begin{document}
\title{A Parametric Family of Subalgebras of the Weyl Algebra \\ 
II. Irreducible Modules}
\author{Georgia Benkart, Samuel A.\ Lopes\thanks{Research funded by the European Regional Development Fund through the programme COMPETE and by the Portuguese Government through the FCT -- Funda\c c\~ao para a Ci\^encia e a Tecnologia under the project PEst-C/MAT/UI0144/2011.}, and Matthew Ondrus}
\date{}
\maketitle

\vspace{-.25 truein}  
\begin{abstract}{An Ore extension over a polynomial algebra $\FF[x]$ is either a quantum plane, a quantum Weyl algebra, or an infinite-dimensional unital associative algebra $\A_h$ generated by elements $x,y$,  which satisfy $yx-xy = h$, where $h\in \FF[x]$.  When $h \neq 0$, the algebras $\A_h$  are subalgebras of the Weyl algebra $\A_1$ and can be viewed as differential operators with polynomial coefficients.   In previous work,  we studied the structure of $\A_h$ and determined its automorphism group $\aut(\A_h)$  and the subalgebra of invariants  under $\aut(\A_h)$.      Here we determine the irreducible $\A_h$-modules.   In a sequel to this paper, we completely describe the derivations of $\A_h$ over any field.} \end{abstract}

\begin{section}{Introduction}\end{section}  
In  \cite{BLO1},  we investigated a family of infinite-dimensional unital associative algebras $\A_h$ parametrized by a polynomial
$h$ in one variable, whose definition is given as follows:
\medskip

\begin{defn}
Let $\FF$ be a field, and let $h \in \FF[x]$.  The algebra  $\A_h$ is  the unital associative algebra over $\FF$ with generators $x$, $y$ and defining relation $yx = xy + h$ (equivalently,  $[y,x] = h$  where $[y,x] = yx-xy$). 
\end{defn}

These algebras arose naturally in the context of Ore extensions over a polynomial algebra $\FF[x]$. Recall that an Ore extension $\A = \DD[y,\sigma, \delta]$  is built from a unital associative (not necessarily commutative) algebra $\DD$ over a field $\FF$, an $\FF$-algebra endomorphism $\sigma$ of $\DD$, and a $\sigma$-derivation of $\DD$, where by a  $\sigma$-derivation $\delta$,  we mean that  $\delta$ is $\FF$-linear and $\delta (rs) = \delta(r)s + \sigma (r) \delta (s)$ holds for all $r,s \in \DD$.  Then  $\A = \DD[y, \sigma, \delta]$ is the algebra generated by $y$ over $\DD$ subject to the relation 
$$yr= \sigma(r)y + \delta(r) \qquad \hbox{\rm for all} \ r \in \DD.$$   Many algebras can be realized as iterated Ore extensions,  and for that reason, Ore extensions  have become a mainstay in associative theory.   Ore extensions inherit 
properties from the underlying algebra $\DD$.   
 For instance,  when $\sigma$ is an automorphism,   then  $\A$ is a free left and right $\DD$-module with basis $\{ y^n  \mid n \ge 0 \}$;  if $\DD$ is left (resp.~right) Noetherian, then $\A$ is left (resp.~right) Noetherian; and 
if $\DD$ is a domain, then $\A$ is a domain.  

The Ore extensions with  $\DD= \FF[x]$ and $\sigma$ an automorphism  have the following description (compare  \cite{AVV87} and \cite{AD97} for a somewhat different division into cases). 
 \begin{lemma}\label{lem:poly}  Assume $\A = \DD[y,\sigma,\delta]$ is an Ore extension with $\DD = \FF[x]$,
a polynomial algebra over a field $\FF$ of arbitrary characteristic, and $\sigma$ an automorphism of $\DD$.   Then $\A$ is isomorphic
to one of the following:
\begin{itemize}
\item[{\rm (a)}] a  quantum plane
\item[{\rm (b)}] a quantum Weyl algebra
\item[{\rm (c)}] a unital associative algebra $\A_h$  with generators $x,y$ and defining relation
$yx = xy + h$ for some polynomial $h \in \FF[x]$.  
\end{itemize}
\end{lemma}
The algebra $\A_h$ is the Ore extension $\DD[y,\mathsf{id}_{\DD}, \delta]$ obtained from taking
 $\DD = \FF[x]$, $h \in \DD$, $\sigma =\mathsf{id}_\DD$, and $\delta: \DD \rightarrow \DD$ to be  the $\FF$-linear derivation with $\delta(r) = r'h$ for all $r \in \DD$,  where $r'$ denotes  the usual derivative of $r$ with respect to $x$.    In particular,  $[y,r] = \delta(r) = r' h$ for all $r \in \DD$.  The algebra  $\A_h$ is a Noetherian domain and a  free left and right $\DD$-module with basis $\{ y^n  \mid n \ge 0 \}$.  Both $\{x^m y^n \mid m,n \in \mathbb Z_{\geq 0}\}$ and  
 $\{y^n x^m  \mid m,n \in \mathbb Z_{\geq 0}\}$ are bases for $\A_h$, and $\A_h$ has Gelfand-Kirillov dimension 2. 
 
 Several well-known algebras have the form $\A_h$ for some $h \in \FF[x]$.  For example, $\A_0$ is the polynomial algebra $\FF[x,y]$; $\A_1$ is the Weyl algebra;  and the algebra $\A_x$ is the universal enveloping algebra of the two-dimensional non-abelian Lie algebra (there is only one such Lie algebra up to isomorphism).   The algebra $\A_{x^2}$ is often referred to as the Jordan plane.  It appears  in noncommutative algebraic geometry (see for example, \cite{SZ94} and \cite{AS95}) and exhibits many interesting features such as being Artin-Schelter regular of dimension 2.   In a series of articles \cite{shirikov05}--\cite{shirikov07-2},  Shirikov has undertaken an extensive study of the automorphisms, derivations,  prime ideals, and modules of the algebra $\A_{x^2}$.  Recent work of Iyudu \cite{I12} has further developed the
 representation theory of $\A_{x^2}$.   Cibils, Lauve, and Witherspoon \cite{CibLauWit09}  have constructed new examples of finite-dimensional Hopf algebras in prime characteristic which are Nichols algebras
using quotients of the algebra $\A_{x^2}$ and cyclic subgroups of their automorphisms.

Quantum planes and quantum Weyl algebras are examples of generalized Weyl algebras, and as such,  have been studied extensively.  There are striking similarities in the behavior of the algebras $\A_h$ as $h$ ranges over the polynomials in $\FF[x]$.  For that reason, we believe that studying them as one family provides much insight into their structure, automorphisms, derivations, and modules.   In \cite{BLO1}, we determined the center, normal elements, prime ideals,  and automorphisms of $\A_h$ and their invariants in $\A_h$.     
In \cite{BLO2},  we determine the derivations of an arbitrary algebra $\A_h$ over any field,  derive expressions for the Lie bracket in the quotient  $\hoch (\A_h) := \der (\A_h) / \inder (\A_h)$
of $\der(\A_h)$ modulo the ideal $\inder(\A_h)$ of inner derivations,  and use  these formulas to understand the structure of the Lie algebra $\hoch (\A_h)$.  In particular, when $\chara(\FF) = 0$, we construct a maximal nilpotent ideal of $\hoch (\A_h)$ and explicitly describe the structure of the corresponding quotient in terms of the Witt algebra (centreless Virasoro algebra) of vector fields on the unit circle.

Our aim in this paper is to give a detailed investigation of the modules for the algebras  $\A_h$ over
arbitrary fields.  In \cite{Bl81}, Block undertook a comprehensive study of the irreducible modules  for the Weyl algebra $\A_1$ and for the universal enveloping algebras of 
$\mathfrak{sl}_2$ and of the two-dimensional solvable Lie algebra  (which is the algebra $\A_x$)
over a field of characteristic zero. (Compare also \cite{AP} for the $\mathfrak{sl}_2$ case.)   Block also considered Ore extensions $\DD[y, \mathsf{id}, \delta]$ over a Dedekind domain $\DD$ of characteristic zero, with the main effort in  \cite{Bl81} directed
towards  investigating irreducible $\DD$-torsion-free modules.   Block's results were extended by Bavula in \cite{Bv99} to more general Ore extensions over Dedekind domains,  and by Bavula and vanOystaeyen in  \cite{BvO97}  to develop a representation theory for generalized 
Weyl algebras over Dedekind domains.  

The  generalized weight $\A_h$-modules over  fields of arbitrary characteristic will form the main focus of the present paper.  Included also will be results on indecomposable
$\A_h$-modules, on primitive ideals of $\A_h$ (that is, the annihilators of irreducible
$\A_h$-modules), and on some combinatorial connections as well.  \medskip

{\it Since the representation theory of polynomial algebras is well developed, we will
assume that $h \neq 0$ throughout the paper.}
\medskip

It is an easy consequence of the relation  $[y,r] = \delta(r)$ for $r\in \DD$  and induction that the following identity  holds in any Ore extension $\DD[y,\mathsf{id}_\DD,\delta]$ for all $n \geq 0$:  
\begin{equation}\label{eq:xynid}  ry^n  = \sum_{j = 0}^n   (-1)^j {n \choose j}  y^{n-j} \delta^j(r). \end{equation}
Using that identity, we obtained  the following description of the center of $\A_h$: 
\begin{thm}\label{T:center}{\rm \cite[Thm.~5.5]{BLO1}}   Let $\centh$ denote the center of $\A_h$.  
\begin{enumerate}
\item[{\rm (1)}]  If $\chara(\FF) = 0$, then $\centh = \FF 1$.
\item[{\rm (2)}]  If $\chara(\FF) = p > 0$, then  $\centh$ is isomorphic to the polynomial algebra $\FF[x^p, z_p]$, where  
\begin{equation}\label{eq:centgen} z_p :=   y (y+h')(y+2h') \cdots (y + (p-1)h') =  y^p -y \frac{\delta^p(x)}{h(x)}, \end{equation} and $'$ denotes the usual derivative.   Moreover $\frac{\delta^p(x)}{h(x)}  \in \centh \cap \FF[x] = \FF[x^p]$.
\end{enumerate}
\end{thm}

\begin{remark} The proof of  Theorem 5.5 in \cite{BLO1} shows that $y$ commutes with $\frac{\delta^p(x)}{h(x)}$, but
since  $\frac{\delta^p(x)}{h(x)}$  is a polynomial in $x$, it commutes with $x$ as well, hence is central in $\A_h$.   \end{remark} 

When  $\chara(\FF) = p > 0$,  it follows from Theorem~\ref{T:center} that  $\A_{h}$ is free of rank $p^{2}$ as a module over its center (see \cite[Prop.~5.9]{BLO1}). This implies that $\A_{h}$ is a polynomial identity ring (e.g.~\cite[Cor.\ 13.1.13\,(iii)]{McR01}).  Applying ~\cite[Thm.\ 13.10.3\,(i)]{McR01},  we can conclude the following:  

\begin{prop}\label{P:findimp}  Assume $\chara(\FF) = p > 0$.   Then all irreducible $\A_{h}$-modules are finite dimensional.  \end{prop} 

\medskip

In Section 2, we review basic facts about modules for Ore extensions over Dedekind domains.   Our approach here follows \cite{Bl81} (see also \cite{Bv99} for results for more general Ore extensions).  For such Ore extensions, the irreducible modules are either generalized weight modules relative to $\DD$ (equivalently, have $\DD$-torsion), or are $\DD$-torsion-free.   We show in Section 3 that for any field $\FF$,  when $h \not \in \FF^*$, the algebra $\A_h$ has  a family of  indecomposable modules of arbitrarily large dimension.  Section 4 is devoted to generalized weight modules for $\A_h$.   In particular, we consider induced generalized weight modules  for $\A_h$, which play a role analogous to Verma modules in the representation theory of semisimple Lie algebras, and also finite-dimensional irreducible modules for $\A_h$.  In Section 5, we determine  the primitive ideals of $\A_h$.  Corollary \ref{C:annchar0} gives an $\A_h$-version  of Duflo's well-known result \cite[Cor.~1]{Du77} on the primitive ideals of enveloping algebras of complex semisimple Lie algebras.

Section 6 is dedicated to the $\chara(\FF) = 0$ case.  Corollary \ref{cor:genwt0} of that section  shows  that the  irreducible generalized weight modules for $\A_h$  are either induced modules or  finite-dimensional quotients of them (compare \cite[Prop.~4.1]{Bl81}).   The classification of the irreducible  generalized weight modules for $\A_h$  when $\FF$ is algebraically closed of characteristic zero  is given in  Corollary  \ref{cor:one-dim}. Part (i) of that corollary may be regarded as the analogue of Lie's theorem for the algebras $\A_h$, and  in fact,  it is Lie's theorem for $\A_x$.  In Section 6.3, we investigate irreducible $\DD$-torsion-free $\A_h$-modules when $\chara(\FF) = 0$ and determine a criterion for when an irreducible $\DD$-torsion-free module for the Weyl algebra $\A_1$ restricts to one for $\A_h$.  When $\chara(\FF) = p > 0$, all irreducible modules are finite dimensional, so $\DD$-torsion-free irreducible modules only exist when $\chara(\FF) = 0$.   When $\FF$ is an algebraically closed field of characteristic $p > 0$, we show in Section 7 that the irreducible $\A_h$-modules have dimension 1 or $p$ and give an explicit description of them in Corollary \ref{C:algclo}.   The expressions for the $\A_h$-action on irreducible modules often entail terms of the form $\delta^k(x)$.   Section 8 presents some interesting combinatorics for these terms phrased  in the language of partitions.

\begin{section}{Modules for Ore Extensions}\end{section}

Assume $\A  = \DD[y,\sigma,\delta]$ is an Ore extension with  $\DD$ a Dedekind domain.  
Let $\EE$ denote  the field of fractions of $\DD$.    Thus $\EE = \Sf^{-1} \DD$ where $\Sf = \DD \setminus \{0\}$.   The localization
$\BB = \Sf^{-1}\A$ is the Ore extension $\BB = \EE[y,\sigma,\delta]$,  where $\sigma$ and $\delta$ have natural extensions to $\EE$.  

Given an $\A$-module $\M$, 
$\tor_{\RR} (\M):=\{ v\in\M \mid r v=0 \ \hbox{\rm for some}\ 0\neq r\in\DD\}$
is an $\A$-submodule called the $\DD$-{\it torsion submodule} of $\M$.  We say that $\M$ is an $\DD$-{\it torsion} (resp.\ $\DD$-{\it torsion-free}) module  if $\tor_{\RR} (\M)=\M$ (resp.\ $\tor_{\RR} (\M)=0$).
If $\M$ is irreducible, then $\Sf^{-1}\M = \BB \otimes_\A \M$ is either 0 or a nonzero
irreducible $\BB$-module.    In the former case, $\M$ has  $\DD$-torsion, and in the latter, $\M$ is  $\DD$-torsion-free. 
Thus, the set $\widehat \A$ of  isomorphism classes  of irreducible $\A$-modules decomposes into two
disjoint subsets,

$$\widehat \A =  \widehat \A (\DD\hbox{-}\mathsf{torsion}) \cup \widehat \A(\DD\hbox{-}\mathsf{torsion\hbox{-}free}).$$

Assume $\M$ is an $\A$-module.   For any ideal $\mathfrak n$  of $\DD$, let 
\begin{equation}\label{eq:wtspaces}  \M_{\mathfrak n} =  \{ v \in \M \mid  \mathfrak n v = 0\}  \quad 
\hbox{\rm and} \quad \M^{\mathfrak n} = \{ v \in \M \mid  {\mathfrak n}^kv = 0 \ \hbox{\rm for some}\  k = k(v)\}.
\end{equation} 

Let $\Max(\DD)$ denote the set of maximal ideals of $\DD$.   An $\A$-module $\M$  is said to be an {\it $\mathsf{R}$-weight module}   (resp. {\it $\mathsf{R}$-generalized weight module})    if  
$\M = \bigoplus_{\mathfrak n \in \Max(\DD)} \M_{\mathfrak n}$ (resp. if $\M = \bigoplus_{\mathfrak n \in \Max(\DD)} \M^{\mathfrak n}$).   When $\DD$ is a Dedekind domain,  the irreducible $\DD$-torsion modules are precisely the irreducible $\DD$-generalized weight modules. 
We present  a proof of this fact next  (compare the arguments in  \cite[Proof of Prop.~4.1]{Bl81} and also \cite [Sec.~4]{Bv99}). 
\medskip

\begin{prop}\label{prop:genwt}
Suppose that $\A = \DD[y,\mathsf{id}_\DD,\delta]$ is an Ore extension with  $\DD$ a Dedekind domain.
\begin{itemize} 
\item[{\rm (i)}]   If $\V$ is an $\A$-module such that $\V = \A u$ for  $u \in \V^{\m}$ and some ideal $\m$ of $\DD$,  then $\V = \V^{\m}$. 
Moreover, if   $\m$ is  $\delta$-invariant and  $u \in \V_\m$, then $\V = \V_\m$. 
\item[{\rm (ii)}] $\widehat \A (\DD\hbox{-}\mathsf{torsion}) =  \widehat \A (\DD\hbox{-}\mathsf{generalized \ weight}).$
\item[{\rm (iii)}]  If $\V$ is an irreducible $\DD$-torsion $\A$-module,  then  $\V = \V^\m$ for some
$\m \in\Max(\DD)$, and when $\m$ is $\delta$-invariant, $\V = \V_\m$.  
\end{itemize}
\end{prop} 

\begin{proof}   Let $\m$ be an ideal of $\DD$ and suppose $\ell \geq 1$. Then by Leibniz's rule,   $\delta(\m^\ell) \subseteq \m^{\ell-1}$  (where $\m^0 = \DD)$.  In case $\m$ is $\delta$-invariant, then $\delta(\m^\ell) \subseteq \m^{\ell}$. 

(i) Assume $\V=\A u$ and $\m^k u = 0$ for some $k \geq 1$. Then by \eqref{eq:xynid}, we have 
$$
\m^{k+n} y^{n}\DD u \subseteq  \sum_{j = 0}^n    y^{n-j} \delta^j(\m^{k+n}) \DD u 
\subseteq  \sum_{j = 0}^n    y^{n-j}  \m^{k+n-j}  u =0.
$$
Thus,  $y^{n}\DD u \subseteq \V^{\m}$ for all $n\geq 0$, which proves that $\V=\A u=\V^{\m}$ (and hence that $\V$ is an $\mathsf{R}$-generalized weight  module if $\m \in \Max(\DD)$).   If $\m$ is $\delta$-invariant, then $\m^{k} y^{n}\DD u \subseteq \sum_{j = 0}^n    y^{n-j}  \m^{k}  u =0$.  Therefore, if $u\in\V_{\m}$,  we can take $k=1$ and obtain $\V=\V_{\m}$, (so that $\V$ is an $\mathsf{R}$-weight module if $\m \in \Max(\DD)$). 

It remains to prove (ii), and then (iii) will be a consequence of that and (i).  The inclusion $$\widehat \A (\DD\hbox{-}\mathsf{generalized \ weight})\subseteq \widehat \A (\DD\hbox{-}\mathsf{torsion})$$ is clear, so we show that if $\V$ is an irreducible $\DD$-torsion $\A$-module,  then $\V$ is an ${\mathsf R}$-generalized weight module. Since $\DD$ is Noetherian, the set
$\left\{ \ann_{\DD}(v) \mid 0\neq v\in\V \right\}$
has a maximal element $\mathfrak p=\ann_{\DD}(u)$, which is nonzero, as $\V$ has $\DD$-torsion. The maximality condition implies that $\mathfrak p$ is a prime ideal of $\DD$. Indeed, if $ab\in\mathfrak p$ and $b\notin\mathfrak p$, then $\mathfrak p=\ann_{\DD}(u)\subseteq\ann_{\DD}(bu)$, so $a \in \ann_{\DD}(bu) =\mathfrak p$.  As $\mathfrak p\neq 0$, $\mathfrak p$ is a maximal ideal of the Dedekind domain $\DD$. Thus, $u\in\V_{\mathfrak p}$ and $\V=\A u=\V^{\mathfrak p}$, by irreducibility and the first part of the proof.  \end{proof} 

\begin{remark} In the remainder of the paper, we will simply say weight module and generalized weight module
with the understanding that always they are with respect to $\DD$. \end{remark}

\begin{lemma} \label{lem: construct1}   Assume   $\A = \DD[y,\mathsf{id}_\DD,\delta]$ is an Ore extension with  $\DD$ a Dedekind domain.    Let  $\m$ be any $\delta$-invariant  ideal of $\DD$,  and let $q$ be a fixed element of $\DD$.   Then the following hold.   
\begin{itemize} 
\item[{\rm (i)}]   The space $\N(\m,q):= \DD/\m$ with the action 
$$s.(r+ \m) = sr+ \m,  \qquad \ \    y.(r+\m) =( qr +  \delta(r)) +\m,$$
for $r,s \in \DD$,   is an $\A$-module.  The $\A$-submodules of $\N(\m,q)$ are the submodules of the form $\mathfrak p/\m$ where $\mathfrak p$ is a $\delta$-invariant ideal of $\DD$ containing $\m$.
\item[{\rm (ii)}]  If $\m$ is a maximal ideal of $\DD$,   then $\N(\m,q) = \N(\m,q)_\m$ is an irreducible weight module. 
\item[{\rm (iii)}] If $\m$ is a maximal ideal of $\DD$ and $n\geq 1$, then $\N(\m^n, q) = \N(\m^n,q)^\m$ is a generalized weight  $\A$-module and  it is uniserial (its submodules are linearly ordered by inclusion), hence it is indecomposable.
\item[{\rm (iv)}]  Assume $\chara(\FF) = 0$.  If $\m$ is a maximal ideal of $\DD$, then
$$\bigcap_{n \geq 1}  \ann_\A\big(\N(\m^n,q)\big) = (0).$$
In particular, if $\DD$ is a finitely generated $\FF$-algebra (e.g. if $\DD = \FF[x]$),   then $\A$ is residually 
finite dimensional (that is to say,  there is a family of ideals of $\A$ of finite co-dimension having trivial intersection).  
\end{itemize}  
\end{lemma}    
\begin{proof}  We leave the verification that $\N(\m,q)$ is an $\A$-module as an exercise for  the reader. It is clear for any $\delta$-invariant ideal $\mathfrak p$ of $\DD$ containing $\m$ that $\mathfrak p/\m$ is an $\A$-submodule of $\N(\m,q)$. Conversely, any 
$\A$-submodule of $\N(\m,q)$ is necessarily an $\DD$-submodule of $\DD/\m$, and thus has the form $\mathfrak p/\m$ for some ideal $\mathfrak p\supseteq \m$ of $\DD$. Given $r\in\mathfrak p$,  we have $y.(r+\m) =(qr +  \delta(r)) +\m$, so $qr +  \delta(r)\in\mathfrak p$. As $qr\in\mathfrak p$ also, it follows that $\delta(r)\in\mathfrak p$, which proves that $\mathfrak p$ is $\delta$-invariant.  Part (ii) follows immediately. 

For part (iii), observe first that whenever $\m$ is $\delta$-invariant, then $\delta(\m^k) \subseteq \m^k$ for all $k \geq 1$,
so that $\m^k$ is $\delta$-invariant.   Thus,  $\N(\m^n,q)$ is  an $\A$-module by (i). Moreover, $\N(\m^n,q)$ is generated by $1+\m^n\in\N(\m^n,q)^{\m}$, so $\N(\m^n,q) = \N(\m^n,q)^\m$ by Proposition~\ref{prop:genwt}.  As $\DD$ is Dedekind, the ideals of 
$\DD$ which contain $\m^n$ are the ideals of the form $\m^{k}$, with $0\leq k\leq n$, and these are all $\delta$-invariant. Thus by (i), the $\A$-submodules of $\N(\m^n,q)$ are $\m^{k}/\m^n$ for $k=0,1,\dots, n$, where $\m^0 = \DD$,  which are obviously linearly ordered by inclusion. This shows that $\N(\m^n,q)$ is uniserial; in particular, it is indecomposable.

For (iv), note first that $\ann_{\DD}\big(\N(\m^n, q)\big)=\m^n$, so 
$$\bigcap_{n\geq 1}\ann_{\DD}\big(\N(\m^n, q)\big) =\bigcap_{n\geq 1}\m^n=(0),$$
because $\DD$ is Dedekind.    Now observe for any nonzero ideal $\mathsf{J}$ of $\A$ that  $\mathsf{J}\cap \DD\neq (0)$. To see this, assume $a=\sum_{i=0}^{k}y^{i}s_{i}$  ($s_{i}\in\DD$ for all $i$) is a nonzero element  of 
minimal $y$-degree in $\mathsf{J}$.   Since $h \neq 0$,  we may take  $r\in\DD$  so that $\delta(r)\neq 0$.  Then by \eqref{eq:xynid},
$$
\mathsf{J} \ni [r, a] =\sum_{i=0}^{k}[r, y^{i}] s_{i}=-ky^{k-1}\delta(r)s_k  + \text{lower order terms in $y$}.
$$
Since $\chara(\FF) = 0$,  the minimality of $k$ forces $k=0$ to hold,  and  $a\in \mathsf{J} \cap \DD$.

If $\bigcap_{n\geq 1}\ann_{\A}(\N(\m^n, q)) \neq (0)$,   then it contains a nonzero $ r \in \DD$.  
But then $r\in\bigcap_{n\geq 1}\ann_{\DD}(\N(\m^n, q))=(0)$.   Hence the ideal  $\bigcap_{n\geq 1}\ann_{\A}(\N(\m^n, q))$ of $\A$ must be trivial, as claimed.

Suppose $\DD$ is a finitely generated $\FF$-algebra. Then the Nullstellensatz implies that $\DD/\m$ is finite dimensional over $\FF$. Since $\N(\m^n, q)$ has finite length, with composition factors isomorphic to $\DD/\m$ as $\DD$-modules, it follows that $\N(\m^n, q)$ is finite dimensional over $\FF$, and so is $\A/\ann_{\A}(\N(\m^n, q))$ for $n \geq 1$.  Since, $\bigcap_{n\geq 1}\ann_{\A}(\N(\m^n, q))=(0)$ we have that $\A$ is residually finite dimensional.
\end{proof} 

\begin{remark} The {\it Jacobson radical} $\mathcal J(\A_h)$ is the intersection of all the primitive ideals
of $\A_h$. If  $a \in {\mathcal J}(\A_h)$, then $1-a$ is invertible. But the invertible elements of $\A_h$ belong to $\FF$ according to  \cite[Thm.~2.1]{BLO1}, so it follows that $a \in \FF$. Since
$\mathcal J(\A_h) \neq \A_h$, it must be that  $a = 0$ and ${\mathcal J}(\A_h) = (0).$ 
Now if $\chara(\FF) = p>0$, then all irreducible modules are finite dimensional by Proposition \ref{P:findimp}, so the ideal $(0)$ is the intersection of ideals of $\A_h$ having finite co-dimension, and 
$\A_h$ is residually finite dimensional. 
\end{remark}

The above results show that special behavior occurs when an ideal of $\DD$ is  invariant under the
derivation $\delta$.    Such ideals are related  to normal elements of $\A$ as the next result shows.    Recall that an element $b \in \A$  is {\it normal} if $\A b = b\A$.  

\begin{lemma}\label{lem:norm-delta}  Assume  $\A = \DD[y,\mathsf{id}_\DD,\delta]$ is any Ore extension,  and let $\m$ be an ideal of $\DD$.  Then $\m$ is $\delta$-invariant if and only if $\m \A = \A \m$.     If $\m = \DD f$ and $\DD$
is commutative,  then  $\m$ is $\delta$-invariant if and only if $f$ is a normal element of $\A$.  
\end{lemma}
\begin{proof}  Suppose that $\m$ is a $\delta$-invariant ideal of $\DD$.  Since $y \m \subseteq \m y + \delta (\m) \subseteq \m \A$ and $\A$ is generated by $\DD$ and $y$, it follows that $\A \m \subseteq \m \A$.  A similar argument shows $\m \A \subseteq \A \m$, so indeed $\m \A = \A \m$.   If $\m$ is an ideal of $\DD$ with $\m \A = \A \m$, then $y \m \subseteq \m \A$.  Thus for any $r\in \m$, $ry + \delta (r) = yr \in \A \m =  \m \A$, and so $\delta (r) \in \m \A - ry \subseteq \m \A$.  Since $\m \A = \bigoplus_{i \ge 0} \m y^i$ and $\delta (r) \in \m \A \cap \DD$, it follows that $\delta (r) \in \m$,  and thus $\m$ is $\delta$-invariant.   Now if $\m = \DD f$ and $\DD$ is commutative,   then $\m$ is $\delta$-invariant if and only if $\A \DD f = f \DD \A$  if and only if $f \A = \A f$ (i.e. $f$ is normal in $\A$). 
\end{proof}

Now assume as before that  $\A = \DD[y,\mathsf{id}_\DD,\delta]$  with  $\DD$ a Dedekind domain, and fix $\m$
an ideal of $\DD$. We can induce the $\DD$-module $\DD/\m$ to an $\A$-module 
\begin{equation}\label{eq:ind}  \UU(\m) :=\A \ot_{\DD} \DD/\m.\end{equation}
 Set  $u_\m: = 1 \ot (1 + \m)\in\UU(\m)$. Since $\UU(\m)=\A u_\m$ and $\m u_\m=0$, Proposition~\ref{prop:genwt}\,(i) implies that $\UU(\m) =\UU(\m)^{\m}$ (and hence that $\UU(\m)$  is a generalized weight $\A$-module if
$\m$ is maximal). Furthermore, if $\m$ is $\delta$-invariant then $\UU(\m) =\UU(\m)_{\m}$ (which is a weight module
when $\m \in \Max(\DD)$).

As $\A$ is a free right $\DD$-module with basis $\{y^k \mid k \in \ZZ_{\geq 0}\}$, it follows (with a slight abuse of notation) that any element of $\UU(\m)$ can be written uniquely as a finite sum $\sum_{k\geq 0} y^k\bar{r}_{k} u_\m$, with $\bar{r}_{k}\in\DD/\m$.  

By the tensor product construction, the $\A$-module $\UU(\m)$ has the following universal property:

\begin{prop}\label{P:universal} 
Let $\V$ be an $\A$-module for $\A = \DD[y,\mathsf{id}_\DD,\delta]$, where  $\DD$ is a Dedekind domain, and 
suppose for some ideal  $\m$ of $\DD$  that $v\in \V_\m$. Then there  is a unique $\A$-module homomorphism $\UU(\m) \to \V$ with $u_\m  \mapsto v$, where $u_\m=1 \otimes (1 + \m)$. If $\V = \A v$, then $\V$ is a homomorphic image of $\UU(\m)$.\end{prop}

\begin{proof}  The map  $\zeta: \A \times \DD/\m \to \V$ given  by $\zeta(a, (r+\m))=arv$  is well defined  because $\m v=0$,  and it is clearly $\DD$-balanced (see \cite[Chap.~9]{Pass}), so it induces an abelian group homomorphism $\A \ot_{\DD} \DD/\m \to \V$, satisfying $a\otimes(r+\m)\mapsto arv$. This is an $\A$-module homomorphism and $u_\m=1 \otimes (1 + \m)  \mapsto v$. The uniqueness is trivial as $\UU(\m)=\A u_\m$, and the remaining statements follow.
\end{proof} 

 \begin{prop}\label{P:Nmq}  Assume $\A = \DD[y,\mathsf{id}_\DD,\delta]$ is an Ore extension  with  $\DD$ a Dedekind domain, and let $\m$ be a $\delta$-invariant ideal of $\DD$.    Assume $\N(\m,q) = \DD/\m$ is as in Lemma \ref{lem: construct1} for some fixed element $q\in \DD$. Then  
 $$\N(\m,q) \cong  \UU(\m)/ \A(y-q)u_\m \cong \A/\big(\A(y-q)+\m\big).$$
\end{prop}

\begin{proof}   By Proposition \ref{P:universal},   there is an $\A$-module map $\zeta: \UU(\m) \to \N(\m,q)$ such that
$\zeta(au_\m) = a(1 + \m)$ for all $a \in \A$.  We claim that  kernel of $\zeta$ is the space  $\mathsf{K} = \A(y-q)u_\m$.   It is easy to check that $\mathsf{K} \subseteq \mathsf{ker}(\zeta)$.   Note that $\{ (y-q)^j \mid  j \in \mathbb Z_{\geq 0}\}$ is a basis for $\A$ viewed as a left $\DD$-module, and $\sum_{j \geq 0} r_j (y-q)^j u_\m \in \mathsf{ker}(\zeta)$ (where $r_j \in \DD$ for all $j$)  if and only if $r_0 u_\m \in \mathsf{ker}(\zeta)$ if and only if $r_0 + \m = \bar 0$.   
But since $r_0 u_\m = 0$ when $r_0 \in \m$, we have  $\mathsf{ker}(\zeta) = \mathsf{K}$, and  $\N(\m,q) \cong \UU(\m)/\ker(\zeta) =\UU(\m)/ \A(y-q)u_\m$, as asserted.

Suppose $\mathsf{J}: = \A(y-q) + \m$.   (This sum is actually a  vector space direct sum, which can be seen 
from the fact that $\{(y-q)^j, j \in \mathbb Z_{\geq 0}\}$ is an $\DD$-basis of $\A$.)    Since $\m$ is $\delta$-invariant, $(y-q)^k r = \sum_{j=0}^k {k \choose j} \delta^j(r)(y-q)^{k-j} \in \mathsf{J}$ for all  $r \in \m$.  Hence, $\mathsf{J}$ is a left ideal of $\A$ and $\A/\mathsf{J} = \A v$ is an $\A$-module generated
by $v = 1+\mathsf{J}$.  Since $\m v = 0$,   there is a homomorphism  $\vartheta: \UU(\m) \to \A/{\mathsf J}$  with
$\vartheta(au_\m) = a+\mathsf{J}$ for all $a \in \A$.    Clearly,  $\A(y-q)u_\m$ is in the kernel,  and $r u_\m$ is in the
kernel for $r \in \DD$  if and only if $r \in \m$.   Thus,  $\UU(\m)/\A(y-q)u_\m \cong  \A/\mathsf{J}$. \end{proof}

 \begin{remark}  Part (i) of Proposition \ref{prop:genwt} and parts (i) and (ii) of Lemma \ref{lem: construct1}  are valid  when $\DD$ is an arbitrary ring.  Thus,  the same induced module $\UU(\m)$ can be constructed,  and the  results in  Proposition \ref{P:universal} and Proposition \ref{P:Nmq}  hold in the more general setting of an Ore extension  $\A = \DD[y, \mathsf{id}_\DD, \delta]$ over any ring $\DD$. \end{remark}

\begin{section}{Indecomposable $\A_h$-modules} \end{section}

{\it For the remainder of this paper, we specialize to the case that the Ore extension is the algebra $\A_h = \DD[y, \mathsf{id}_\DD, \delta]$, where $\DD = \FF[x]$ and $\delta(r) = r'h$ for all $r \in \DD$.}
\medskip

In this section, we use the modules $\N(\m^{n+1},q)$ for $n \geq 0$ from Section 2  to show that for any field $\FF$, if   $h \not \in \FF^*$,  then $\A_h$  can have an indecomposable module of dimension $n+1$ for any $n \geq 0$.  To provide an explicit description of the action of $\A_h$,  we will use a modified version of the usual $k$th derivative $f^{(k)}$ of $f \in \FF[x]$ when $\chara(\FF) = p > 0$, which we introduce next.

For any $k \in \mathbb Z_{\geq 0}$,  we write its $p$-adic expansion as  $k=\sum_{i  \geq 0}k_i  p^i$, where $0 \leq k_i < p$ for all $i$.   It is well known that in characteristic $p>0$, if $k,\ell \in \ZZ_{\geq 0}$, then 
$${\ell  \choose k}   =   \prod_{i \geq 0} {\ell_i \choose k_i}.$$    
Set 
\begin{equation}\label{eq:xl[k]} (x^\ell)^{[k]}  =  \Bigg(\prod_{i\geq 0}  \ell_i(\ell_i-1) \cdots (\ell_i-k_i+1) \Bigg)\,x^{\ell-k}. \end{equation} 
When $k_i= 0$, we interpret the product $  \ell_i(\ell_i-1) \cdots (\ell_i-k_i+1)$
as being 1.   This ``$p$-adic'' derivative can be extended linearly  to arbitrary polynomials $f \in \FF[x]$.  We write $f^{[k]}$ for the result and note that $f^{[0]} = f$.  
 
\medskip 

\begin{prop} \label{prop: construct}   Assume $h \not \in \FF^*$ and   $\m = \DD(x-\lambda)$, where $h(\lambda) = 0$.  Let  $q$ be a fixed element of $\DD$.   Then for all $n \geq 0$,   the module $\N(\m^{n+1}, q)$   is an indecomposable $\A_h$-module of dimension $n+1$  with basis $\{v_j  \mid j = 0,1,\dots, n\}$ such that  for each $j$, 
\begin{itemize}
\item[{\rm (i)}]   $x. v_j   =  \lambda v_j  + v_{j-1}$,   where $v_{-1} = 0$;
\vspace{-.15 truein} 
\item[{\rm (ii)}]  $y. v_j   = q.v_j +(n-j)h.v_{j+1} = q.v_j + (n-j)\displaystyle{\sum_{\ell = 0}^j  \eta_{j+1-\ell} v_\ell}$,   
where
\begin{equation}\label{eq:etal} \eta_k = \begin{cases}  \displaystyle{ \frac{h^{(k)}(x) \mid_{x = \lambda}}{k !}} & \qquad \hbox{\rm if} \ \ \chara(\FF) = 0, \\
 \displaystyle{ \frac{h^{[k]}(x) \mid_{x = \lambda}}{\prod_{i \geq 0} k_i !}} & \qquad \hbox{\rm if} \ \ \chara(\FF) = p > 0,  \end{cases} \end{equation} 
for all $k  \geq 0$,    and $q.v_j$ is computed
using \eqref{eq:fvk} below.  
\end{itemize}
 \end{prop}

\begin{proof}  Since $\m$ is a maximal, $\delta$-invariant ideal of $\DD = \FF[x]$, 
we know  by Lemma \ref{lem: construct1}\,(iii)  that $\N(\m^{n+1},q)$ is an indecomposable generalized weight $\A_h$-module for all $n \geq 0$.

To simplify the notation in the remainder of the proof,  set $\mathfrak{p} = \m^{n+1}$.        Let
$v_j: = (x - \lambda)^{n-j} + \mathfrak{p}$  for $j = 0,1,\dots, n$,  and set $v_{j} = 0$ if $j < 0$.      Then
$$(x-\lambda). v_j  =  (x - \lambda)^{n-(j-1)} + \mathfrak{p} = v_{j-1},$$
so that $x. v_j = \lambda v_j + v_{j-1}$ holds for all $j$ as in (i).
Arguing by  induction,  we have
$$x^\ell.v_j = \sum_{k = 0}^\ell  {\ell \choose k} \lambda^{\ell-k} v_{j-k}$$
for all $\ell  \geq 0$.   Hence,  it follows that for any polynomial $f = f(x) \in \DD$,  
\begin{equation}\label{eq:fvk} f.v_j = \begin{cases}  \displaystyle{  \sum_{k \geq 0} \frac {f^{(k)}(x) \mid_{x = \lambda}}{k !}\,v_{j-k}}  & \qquad \hbox{\rm if} \ \ \chara(\FF) = 0, \\
 \displaystyle{  \sum_{k \geq 0} \frac {f^{[k]}(x) \mid_{x = \lambda}}{\prod_{i \geq 0} k_i !}\,v_{j-k}} & \qquad \hbox{\rm if} \ \ \chara(\FF) = p > 0, \end{cases}
  \end{equation} 
where  $f^{(0)} = f = f^{[0]}$.    In particular, 
\begin{eqnarray*}  h.v_j    &=& \sum_{k= 1}^j  \eta_k  v_{j-k} = \sum_{k = 0}^{j-1} \eta_{j-k} v_{k}, \end{eqnarray*} 
where $\eta_k$ is as in \eqref{eq:etal},   and $\eta_0 = 0$ since $h(\lambda) = 0$. 

Now 
\begin{eqnarray*} y.v_j  &=&  y. \bigg(  (x-\lambda)^{n-j} + \mathfrak{p}\bigg) =  q(x-\lambda)^{n-j}   + \delta\big((x-\lambda)^{n-j}\big) + \mathfrak{p} \\
&=& q.v_j  +  (n-j)h.v_{j+1}  \\
&=& q.v_j +  (n-j)\sum_{k = 0}^{j} \eta_{j+1-k} v_{k}, \end{eqnarray*} 
where $q.v_j$ can be computed using \eqref{eq:fvk},  to give (ii).    \end{proof}
 
 \begin{remark} In the preceding result,  the space  $\N_j: = \spann_\FF\{v_0,v_1, \dots, v_j\}$
 is an $\A_h$-submodule of $\N(\m^{n+1},q)$ for each $j = 0,1,\dots, n$.   Set $\N_{-1} = \m^{n+1}$.   If $\overline v_j = v_j + \N_{j-1}$, then for $j = 0,1,\dots, n$, 
 we have $x.\overline v_j= \lambda \overline v_j$ and $y. \overline v_j= \mu_j \overline v_j$, where
 $\mu_j = q(\lambda) + (n-j) \eta_1$.   Therefore,  $\N_j/\N_{j-1} = \FF\overline v_j \cong \V_{\lambda,\mu_j}$ in the notation used in Theorem \ref{thm:fd} and  Corollary \ref{C:algclo} below. 
\end{remark}

\begin{section}{Generalized Weight Modules for $\A_h$}\end{section}

For the algebra $\A_h = \DD[y, \mathsf{id}_\DD, \delta]$, a maximal ideal  $\m = \DD f$ of  $\DD = \FF[x]$   is $\delta$-invariant if and only if  $f$ divides $\delta(f) = f' h$.   Since $f$ is a prime polynomial, the only way that can happen when $\chara(\FF) = 0$ is if $f$ is a prime factor of $h$. Therefore,  the $\delta$-invariant maximal ideals are exactly the ideals generated by the prime factors of $h$ when $\chara(\FF) = 0$.   When $\chara(\FF) = p > 0$, then $f$ divides $\delta(f) = f' h$ exactly when $f$ is a prime factor of $h$ or when $f' = 0$.   In the latter case, $f \in \FF[x^p]$. 
(This could also be deduced using Lemma \ref{lem:norm-delta} above and Theorem 7.3 of \cite{BLO1},  which gives a complete description of all the normal elements of $\A_h$.)
We record these facts for later use.   \medskip
 
\begin{lemma}\label{lem:delta-inv}  Assume  $\m = \DD f$ is a $\delta$-invariant maximal ideal of  $\DD = \FF[x]$,
where $\delta(r) = r'h$ for all $r \in \DD$.  
If $\chara(\FF) = 0$, then $f$ is a prime factor of $h$;  if $\chara(\FF) = p > 0$, then either $f$ is a prime factor of $h$
or $f \in \FF[x^p]$.  \end{lemma}

\begin{subsection} {Induced $\A_h$-modules}  \end{subsection}   

Assume $\m$ is an ideal of $\DD$, not necessarily $\delta$-invariant.  The induced $\A_h$-module, 
$\UU(\m) := \A_h \ot_\DD  \DD/\m = \A_h u_m$,  where  $u_\m: = 1 \ot (1 + \m)$,    
has a basis 
$$\{y^k x^\ell u_\m  =  y^k \ot (x^\ell + \m) \mid 0 \leq \ell < \mathsf{dim}(\DD/\m), \  k \in \ZZ_{\geq 0}\}$$
with $\A_h$-action given by 
\begin{eqnarray}\label{eq:Ahact} 
 x . y^n x^\ell u_\m  &= &  \sum_{j=0}^{n} (-1)^j {n \choose j}  y^{n-j}\delta^j(x)x^\ell u_\m,    \\
y. y^n x^\ell u_\m  &= &  y^{n+1} x^\ell u_\m. \nonumber \end{eqnarray}   
Then for $r \in \FF[x]$, 
\begin{equation}\label{eq:zero}  y^k r u_\m = 0 \ \ \hbox{\rm if and only if} \ \   r \in \m. \end{equation}      
 Since $\UU(\m) = \A_h u_\m$ and $\m u_\m = 0$,  by Proposition \ref{prop:genwt}\,(i)  we have  $\UU(\m) = \UU(\m)^\m$,  
and when $\m$ is $\delta$-invariant,  $\UU(\m) = \UU(\m)_\m$.
\smallskip

We assume now that $\m \in \Max(\DD)$ so that $\m = \DD f$ for some prime polynomial $f \in \DD$,
and consider first  the following case:
\bigskip

\noindent \textbf{$\boldsymbol f$ is a factor of $\boldsymbol h$:}   Since $\m$ is $\delta$-invariant when $f$ is a factor of $h$,  $\UU(\m) = \UU(\m)_\m$.    Lemmas 6.1 and 7.1 of \cite{BLO1} show that $[\A_h,\A_h] \subseteq h \A_h = \A_h h$.
Thus, for any $a,b \in \A_h$ and $w \in \UU(\m)$, we have

\begin{equation}\label{eq:xyw} baw = abw +[b,a]w=abw, \end{equation}   
and  $a  \UU(\m)$ is an $\A_h$-submodule
of $\UU(\m)$ for any $a \in \A_h$.    

 If $a = \sum_{i \geq 0} y^i  r_i$ and  $b = \sum_{i \geq 0} y^i s_i$, where $r_i, s_i \in \DD$,  then $a \UU(\m) = b \UU(\m)$ if and only if
$r_i - s_i \in \m$ for all $i \geq 0$.   In particular,  $a \UU(\m) = 0$ if and only if  $r_i \in \m$ for all $i$. 
Hence $\m[y]$ annihilates $\UU(\m)$,  and the action of $\A_h$  on  $\UU(\m)$ is the same as the action of the  commutative polynomial algebra  $\PR_\m = (\DD/\m)[y]  \cong  \DD[y]/\m[y]$.     

Let $\W$ be a submodule of $\UU(\m)$,  and set 
\begin{equation}\label{eq:JW} \mathsf{J_W} = \bigg \{ \bar a = \sum_{i \geq 0} y^i \bar r_i  \in \PR_\m \,\bigg |\,  \bar a \UU(\m) 
\subseteq \W \bigg\}\end{equation}   Then $\mathsf{J_W}$ is an ideal of  the PID $\PR_\m$, 
and we may assume  $\mathsf{J_\W} = \PR_\m  \bar g$ for some monic polynomial  $\bar g =  \sum_{i \geq 0} y^i \bar g_i  \in \PR_\m$.   The
 map $\PR_\m \rightarrow \UU(\m)/\W$ given by $\bar a  \mapsto  \bar a(u_\m + \W)$ is onto and has kernel $\PR_\m  \bar g$.
 Thus,   $ \UU(\m)/\W \cong \PR_\m /\PR_\m \bar g$, which has dimension $\degg(f)\degg(\bar g)$ over $\FF$, and 
 $\UU(\m)/\W$ is irreducible when $\bar g$ is a prime polynomial in $\PR_\m$.  
   
Conversely, if  $\bar g \in \PR_\m$, then $\bar g \UU(\m)$ is a submodule of $\UU(\m)$ and
$\UU(\m)/\bar g \UU(\m)  \cong  \PR_\m/\PR_\m \bar g$.   
When $\bar g$ is a monic prime polynomial in $\PR_\m$,  the quotient 
\begin{equation}\label{eq:lmgdef} \Lmg: = \UU(\m)/ \bar g \UU(\m)  \cong \PR_\m/ \PR_\m {\bar g}\end{equation}  is irreducible, and by the preceding paragraph,  every irreducible quotient of $\UU(\m)$ has
this form.   Any irreducible generalized weight $\A_h$-module $\V = \V^\m$ must be a weight module,  $\V = \V_\m$ by Proposition \ref{prop:genwt}\,(iii),  since $\m$ is $\delta$-invariant.  Moreover, since  $\V$ is a homomorphic image of $\UU(\m)$, it is isomorphic to some irreducible quotient of $\UU(\m)$.  Hence,  $\V \cong \Lmg$ for some monic prime polynomial $\bar g$ of $\PR_\m$.  
\bigskip

\noindent \textbf {$\boldsymbol f$ is not a  factor of $\boldsymbol h$:}  \ \  Assume now   that $\underline {\chara(\FF) = 0}$ and $f$ is not a factor of $h$.   Let  $\W$ be a nonzero submodule of $\UU(\m)$.   Let  $0 \neq w = \sum_{k=0}^n y^k r_k u_\m$ be an element 
of minimal degree in $y$ lying in $\W$, where $r_k \in \DD$ for all $k$ and $\degg r_k < \degg f.$ Then  $f$ does
not divide $r_n$ by the minimality assumption. 
Applying $f$ we have

$$fw = \sum_{k=0}^n \sum_{j=0}^k(-1)^j {k \choose j}  y^{k-j} \delta^j(f) r_k u_\m \in  \W.$$
Since $\delta^0(f) = f$, and $fr_ku_\m = 0$ for all $k = 0,1,\dots, n$,   the element $fw$ has smaller degree in $y$, and so must be 0.   Now if $n \geq 1$, this implies that $ny^{n-1}\delta(f)r_nu_\m = 0$.     Since $\delta(f)r_n$
is not divisible by $f$ and $\chara(\FF) = 0$,   we have arrived at a contradiction.  Hence, any nonzero element of minimal
$y$-degree in $\W$ must have the form $w = r_0 u_\m$.    But since $\DD/\m$ is a field, there
exists an $s\in \DD$ so that  $s r_0 \equiv 1 \modd \m$.   Thus,  
$u_\m =  sr_0 u_\m = sw \in \W$.   Consequently, $\UU(\m) = \A_h u_\m \subseteq \W$, and
$\UU(\m)$ is an irreducible generalized weight module for $\A_h$.      We summarize what we have just shown.
\bigskip

\begin{thm}\label{T:induce}  Let  $\m =  \DD f$ be the  maximal ideal 
of $\DD = \FF[x]$ generated by the prime polynomial $f$, and let $\UU(\m): = \A_h \ot_\DD \DD/\m$ be the $\A_h$-module induced from
the irreducible $\DD$-module $\DD/\m$.    Then the following hold:
 \begin{itemize}  
\item[{\rm (i)}]  $\UU(\m) = \UU(\m)^\m$ is a generalized weight
module for $\A_h$.   If $\m$ is $\delta$-invariant,  then $\UU(\m) = \UU(\m)_\m$ is a weight module for
$\A_h$. 
\item[{\rm (ii)}]     If $f$ is a factor of $h$, then $\UU(\m) = \UU(\m)_\m$.
For any monic prime polynomial $\bar g \in \PR_\m =(\DD/\m)[y]$,  the quotient 
$\Lmg = \UU(\m)/\bar g \UU(\m)$  is an irreducible weight 
 $\A_h$-module of dimension $\deg(f)\deg(\bar g)$ over $\FF$, and any irreducible generalized
 weight module $\V = \V^\m$ for $\A_h$ is isomorphic to 
$\Lmg$ for some monic prime polynomial $\bar g \in \PR_\m$.    
 \item[{\rm (iii)}] If $\chara(\FF) = 0$, and $f$ is not a factor of $h$,  then $\UU(\m) = \UU(\m)^\m$
 is an irreducible generalized weight module for $\A_h$.     \end{itemize}
\end{thm} 

\begin{subsection}{Finite-dimensional $\A_h$-modules} \end{subsection}

Let $\V$ be an irreducible weight module for  $\A_h$ such that
$\V = \V_\m$ for some $\delta$-invariant maximal ideal $\m = \DD f$ of $\DD$.  
 Recall that the ideal $\m$ is $\delta$-invariant  if and only if  $f$ divides $
\delta(f) = f'h$, which says that either $f$ is a prime factor of $h$ or else
$\chara(\FF) = p > 0$ and $f \in \FF[x^p]$ (as in Lemma \ref{lem:delta-inv}).   
Since $\V$ is a homomorphic
image of $\UU(\m)$ by Proposition \ref{P:universal},   Theorem \ref{T:induce}(ii)  shows that $\V$ is finite dimensional whenever  $f$ is a prime factor of $h$.    Since by Proposition \ref{P:findimp},   {\it any}
irreducible module is finite dimensional when $\chara(\FF) = p > 0$,
an  irreducible $\A_h$-module $\V$ such that $\V= \V_\m$ and $\m$ is
$\delta$-invariant is always finite dimensional.    Next we explore the converse. 
    \medskip

\begin{lemma}\label{lem:ann} 
Assume  $\M$ is  any  finite-dimensional irreducible $\A_h$-module. Then there exists a monic prime polynomial $f \in \DD$ so that $\M = \M^\m$ for $\m = \DD f$.  Either $\m$ is $\delta$-invariant and $\M = \M_\m$,  or $\chara(\FF) = p > 0$ and $\ann_\DD(\M) = \m^p  = \DD f^p$,   where $f \not \in \FF[x^p]$.   \end{lemma}  

\begin{proof}  Since $\M$ has $\DD$-torsion and is irreducible,  $\M = \M^\m$ for some maximal ideal $\m = \DD f$ generated by a monic prime polynomial $f \in \DD$ by  Proposition  \ref{prop:genwt}.   As  $\M$ is finite dimensional, there is a least integer $k \geq 1$ so that $\m^k  \M = 0$.   Hence $\m^k = \ann_\DD(\M)$.    Since for any $v \in \M$,  we have $0 = y f^kv - f^k y v  = \delta(f^k) v$,  it must be that  $\delta(f^k) \in \m^k = \DD f^k$.   But this says, $f^k$ divides $k f^{k-1} f' h$, and hence that $f$ divides $k f' h = k \delta(f)$.     If  $f$ divides $\delta(f)$, then $\m = \DD f$ is $\delta$-invariant and $\M = \M_\m$.   If that is not the case, then $\chara(\FF) = p > 0$,  $f' \neq 0$,  and $k \equiv 0 \modd p$ must hold.

Assume now that $\chara (\FF) = p>0$.  Since $\m^p = \DD f^p$ and $f^p \in \FF[x^p] \subseteq \centh$, it follows that $\m^p \M$ is an $\A_h$-submodule of $\M$.  Because $\M$ is irreducible, either $\m^p \M = 0$ or $\m^p \M = \M$.  If $\m^p \M = \M$, then $\m^{2p} \M = \m^p (\m^p \M) = \m^p \M = \M$, and (proceeding inductively) $\m^{(n+1)p} \M = \m^{np} ( \m^p \M) = \m^{np} \M = \M$.  Since some power of $\m$ must annihilate $\M$, it is necessarily the case that $\m^p \M = 0$.
\end{proof} 

\begin{remark} When $\chara(\FF) = p > 0$ and $\lambda$ is not a root of $h(x)$, the irreducible $\A_h$-modules
$\M =\Lmgb$ appearing in Lemma \ref{L:Vlb}  below have the property that $\ann_\DD(\M) = \m^p$
where $\m = \DD(x-\lambda)$.   As we show in Corollary \ref{C:algclo},  they,  along with the one-dimensional 
modules,  are the only irreducible $\A_h$-modules when $\FF$ is algebraically closed of
 characteristic $p$.  \end{remark}

\begin{section}{Primitive ideals of $\A_h$}  \end{section}

Recall that a {\it primitive ideal} is the annihilator of an irreducible module;  in other words, it is the kernel of an irreducible representation. A ring is primitive if it has a faithful irreducible module.  In any ring, primitive ideals are prime,  and maximal ideals are primitive, but the converses of  these statements generally fail to be true. For the universal enveloping algebra of a finite-dimensional nilpotent Lie algebra over a field of characteristic $0$, \cite[Prop.~4.7.4]{Dix} shows that all primitive ideals are maximal. We will see below that this does not hold for $\A_{h}$.  In fact,  if $\chara (\FF)=0$,  then $\A_{h}$ has faithful irreducible modules.

In  \cite[Thm.\ 7.7]{BLO1} we determined the height-one prime ideals of $\A_{h}$ and noted in \cite[Remark 7.9]{BLO1} that the maximal ideals of $\A_{h}$ are the prime ideals of height two. (The \emph{height} of a prime ideal is the largest length of a chain of prime ideals contained in it, or is said to be $\infty$ if no bound exists.)  In Proposition \ref{P:prim} below,  we determine  the  primitive ideals of $\A_h$.   Our argument uses the following result, which holds quite generally. 
\medskip

\begin{lemma} \label{lem:finDimPrimMax}
Let $\A$ be an associative $\FF$-algebra.   Suppose $\M$ is a finite-dimensional irreducible $\A$-module, and let $\mathsf P = \ann_{\A}(\M)$.  Then $\mathsf{P}$ is a maximal ideal of $\A$, and $\A / \mathsf{P} \cong \mathsf{End}_{\mathsf{D}}(\M)$, where $\mathsf{D} = \mathsf{End}_{\A}(\M)$.
\end{lemma}

\begin{proof}
The representation $\A \to \mathsf{End}_{\FF}(\M)$ induces an injective homomorphism 
\begin{equation}\label{E:jac-2}
\A/\mathsf{P}\hookrightarrow\mathsf{End}_{\FF}(\M).
\end{equation}
Let $\mathsf{D}=\mathsf{End}_{\A}(\M)$. By Schur's Lemma, $\mathsf{D}$ is a division ring containing $\FF\mathsf{id}_{\M}$,  and  $\M$ is finite dimensional over $\mathsf{D}$.  The image of \eqref{E:jac-2} is contained in $\mathsf{End}_{\mathsf{D}}(\M)$,  and the Jacobson Density Theorem implies that $\A/\mathsf{P}\cong\mathsf{End}_{\mathsf{D}}(\M)$. Hence $\A/\mathsf{P}$ is simple,  and $\mathsf{P}$ is maximal.
\end{proof}
 
\begin{prop}\label{P:prim}
An ideal $\mathsf{P}$ of $\A_{h}$ is primitive if and only if $\mathsf{P}$ is maximal,  or $\chara(\FF)=0$ and $\mathsf{P}=(0)$. In particular, if $\chara(\FF)=0$,  then $\A_{h}$ is a primitive algebra,  and all  infinite-dimensional irreducible $\A_{h}$-modules are faithful.
\end{prop}
\begin{proof}
As mentioned earlier, any maximal ideal is primitive. Let $\mathsf{P}$ be a primitive ideal of $\A_{h}$,  and let $\M$ be an irreducible $\A_{h}$-module with annihilator $\mathsf{P}$. 

If  $\chara(\FF)=p>0$, then by Proposition~\ref{P:findimp}, $\M$ is finite dimensional, 
and Lemma \ref{lem:finDimPrimMax} implies that $\mathsf{P}$ is maximal.  
 
Now assume $\chara(\FF)=0$. If $\mathsf{P}\neq (0)$,  then $\mathsf{P}$ contains a height-one prime ideal. By \cite[Thm.\ 7.7]{BLO1},  we deduce that $\mathsf{P}$ contains a prime factor of $h$.  But then  $h\in\mathsf{P}$,  and $\A_{h}/\mathsf{P}$ is commutative, as $[y, x]\in\mathsf{P}$. Hence $\M\cong\A_{h}/\mathsf{P}$,  and $\mathsf{P}$ must be a maximal ideal. In particular, in this case $\A_{h}/\mathsf{P}$ is finite dimensional (it is a finitely generated field extension of $\FF$),  and thus $\M$ is also finite dimensional. This shows that if $\M$ is an infinite-dimensional irreducible $\A_{h}$-module,  then $\mathsf{P}=\ann_{\A_{h}}(\M)=(0)$ and $\M$ is faithful. It remains to show that $(0)$ is a primitive ideal when $\chara(\FF) = 0$.  But that follows from the existence of  infinite-dimensional irreducible $\A_{h}$-modules. Indeed, by Theorem~\ref{T:induce}\,(iii), if $\chara(\FF) = 0$ and $f\in\DD$ is a prime polynomial
which is not a factor of $h$ (e.g. if $f=x-\lambda$ with $h(\lambda)\neq 0$), then the induced generalized weight module $\UU(\m)$ for $\m = \DD f$  is irreducible and infinite dimensional, thus faithful.    \end{proof}

\begin{remark} Iyudu \cite{I12} has shown that this result holds for the algebra $\A_{x^2}$ over
algebraically closed fields of characteristic 0.   It should be noted
that the roles of $x$ and $y$ in \cite{I12} are reversed,  and the ideal (0) needs to be added to 
statement of Corollary 5.4 in \cite{I12}.   \end{remark}

In the proof of Proposition \ref{P:prim}, we have seen that when $\chara(\FF) = 0$
and $\mathsf{P}$ is a nonzero primitive ideal, then $\mathsf{P}$ is a maximal ideal containing a  prime factor $f$  of $h$.      Let  $\m = \DD f$.   Since $\A_h/\mathsf{P} = (\A_h/\mathsf{P})_\m$
is an irreducible weight module,  by Theorem \ref{T:induce}\,(ii)  there exists a monic  
prime  polynomial  $\bar g$ in $\PR_\m = (\DD/\m)[y]$  such that
$\A_h/\mathsf{P}\cong \Lmg$.   Hence, $\mathsf{P}$ is the annihilator of 
one of the finite-dimensional irreducible modules $\Lmg$.    We have the following analogue
of Duflo's result on the primitive ideals of the universal enveloping algebra of a finite-dimensional complex semisimple Lie algebra (see \cite{Du77}).  

\begin{cor}\label{C:annchar0} \begin{itemize} \item[{\rm (a)}] Assume $\chara(\FF) = 0$ and $h \not \in \FF^*$.   A primitive ideal of $\A_h$ is $(0)$ or is the annihilator of an irreducible module $\Lmg$ for $\m = \DD f$, where $f$ is a prime factor of $h$, and $\bar g$ is a monic prime polynomial of $\PR_\m = (\DD/\m)[y]$.  The primitive ideal $(0)$ is the annihilator of $\UU(\m)$  for any maximal ideal $\m$ of $\DD$ which is not $\delta$-invariant.
  \item[{\rm (b)}]  Over any field $\FF$,  if  $\m = \DD f$, where $f$ is a prime factor of $h$, and if 
$g = \sum_{j \geq 0} y^j g_j \in \A_h$ (where $g_j \in \DD$ for all $j$)  has the property that $\bar g  = \sum_{j \geq 0} y^j \bar g_j$ is a monic prime polynomial  in $\PR_\m$,  then  $\ann_{\A_h}(\Lmg) = \A_h g + \A_h \m$.  \end{itemize}    \end{cor} 

\begin{proof}   Only part (b)  remains to be shown.   
Clearly,  $\A_h g + \A_h\m \subseteq 
\ann_{\A_h}(\Lmg)$.      For the other direction,  assume $a = \sum_{j \geq 0}y^j r_j \in 
\ann_{\A_h}(\Lmg)$.    Since the action of $a$ on $\Lmg = \UU(\m)/\bar g \UU(\m)$ is
the same as the action of $\bar a = \sum_{j \geq 0}y^j \bar r_j $ on $\PR_\m/\bar g \PR_\m$,   it must be that $\bar a$ is divisible
by $\bar g$.    Thus, there exists a $b = \sum_{j \geq 0} y^j b_j \in \A_h$ (with $b_j \in \DD$ for all $j$)  so that $\bar a = \bar b \bar g$ in $\PR_\m$,  where $\bar b = \sum_{j \geq 0} y^j \bar b_j$.  Hence $a-bg
\in \m[y] = \A_h \m$,    and  $a \in \A_h g + \A_h \m$.       \end{proof}

\begin{cor}\label{C:anniso}  Assume $\LL(\m_i,\bar g_i)$ for $i = 1,2$ are two irreducible $\A_h$-modules  as in \hbox{\rm Corollary \ref{C:annchar0}\,(b)}.  Then the following are equivalent:
\begin{itemize} 
\item[{\rm (a)}] $\LL(\m_1,\bar g_1) \cong \LL(\m_2,\bar g_2)$.
\item[{\rm (b)}] $\m_{1}=\m_{2}$,  and $\bar g_1=\bar g_2$ as polynomials in $\PR_{\m_{1}}=\PR_{\m_{2}}$.
\item[{\rm (c)}] $\ann_{\A_h}(\LL(\m_1, \bar g_1)) =  \ann_{\A_h}(\LL(\m_2, \bar g_2))$.
\end{itemize}
\end{cor} 
\begin{proof} For $i=1,2$, the maximal ideal  \  $\m_i$ \  is determined by \ $\m_i\, =\, \ann_\DD\big(\LL(\m_i,\bar g_i)\big)=$\newline  $\DD\cap \ann_{\A_{h}}\big(\LL(\m_i,\bar g_i)\big).$   In particular, if the generator $f_i$ of $\m_i$
 is assumed to be monic, it is uniquely determined. Then $\bar g_i$ is the monic prime polynomial in $(\DD/\m_{i})[y]$ which annihilates $\LL(\m_i,\bar g_i)$.  Equivalently, it is the generator of  $\ann_{\A_h}(\LL(\m_i, \bar g_i))/\A_{h}\m_{i}$ as an ideal of $\PR_{\m_{i}}$. Since $\ann_{\A_h}(\LL(\m_i, \bar g_i))$ is determined by the isomorphism class of $\LL(\m_i,\bar g_i)$ we have $(a)\implies (c)$, and $(c)\implies (b)$ by the above. Finally, since $\m_i$ and $\bar g_i$ determine
$\LL(\m_i,\bar g_i)$, we have $(b)\implies (a)$. 
  \end{proof}   

The equivalence of (a) and (c) in the previous corollary  is a general phenomenon.  We include a 
proof of this equivalence in a very general context next for the convenience of the reader,  and also because the following proposition can be used to deduce information about  the primitive ideals in Corollary \ref{C:algclo} below. 
  
 \begin{prop}\label{P:annVW} 
Let $\A$ be an associative $\FF$-algebra,  and let $\V,\W$ be finite-dimensional irreducible $\A$-modules. Then $\V\cong\W$ if and only if $\ann_\A(\V)=\ann_\A(\W)$.  Thus,  the isomorphism classes of finite-dimensional irreducible $\A$-modules are in bijection with the maximal ideals of $\A$ of finite co-dimension.
\end{prop}

\begin{proof}
Assume  $\phi: \V\rightarrow\W$ is a surjective $\A$-homomorphism.  Then 
$$\ann_{\A}(\W)= \ann_{\A}(\phi(\V)) \supseteq \ann_\A(\V),$$ so $\ann_{\A}(\W)\supseteq \ann_{\A}(\V)$,  and equality holds if $\phi$ is an isomorphism.

Conversely, suppose $\V$ is a finite-dimensional irreducible  $\A$-module,  and let  $\mathsf{P} = \ann_{\A}(\V)$.  Lemma \ref{lem:finDimPrimMax} implies that $\mathsf{P}$ is maximal and of finite co-dimension in $\A$.  Furthermore, if $\W$ is another irreducible $\A$-module with $\ann_\A(\W)=\ann_\A(\V)=\mathsf{P}$, then $\V$ and $\W$ are two irreducible modules over the simple Artinian ring $\mathsf{\End}_{\mathsf{D}}(\V) \cong \A / \mathsf{P}$, where $\mathsf{D}=\mathsf{\End}_{\A}(\V)$. But this ring has only one irreducible module up to isomorphism.  Thus $\V\cong\W$ as $\A/\mathsf{P}$-modules, hence also as $\A$-modules.   \end{proof}

\begin{section}{Irreducible $\A_h$-modules when $\chara(\FF) = 0$} \end{section} 

\begin{subsection}{Irreducible generalized weight modules for $\A_h$} \end{subsection} 

It is an immediate consequence of  Theorem \ref{T:induce} and the fact that a maximal ideal $\m = \DD f$ is $\delta$-invariant if and only if $f$ divides $h$ when $\chara(\FF) = 0$ that the following holds.

\medskip
\begin{cor}\label{cor:genwt0} Assume $\chara(\FF) = 0$.  Let $\V$  be an irreducible generalized
weight $\A_h$-module.  Then $\V = \V^\m$ for  some maximal ideal $\m = \DD f$ of $\DD$
 generated by a prime polynomial $f$.
  \begin{itemize} 
 \item[{\rm (i)}]  If $f$ is a factor of $h$,  then $\V = \V_\m$  and
 $\V \cong  \Lmg = \UU(\m)/\bar g \UU(\m)$ for some monic prime
 polynomial $\bar g \in (\DD/\m)[y]$.   
\item[{\rm (ii)}] If $f$ is not a  factor of $h$,  then $\V$ is isomorphic
to the induced module  $\UU(\m) = \A_h \ot_\DD \DD/\m$.  
\end{itemize}
\end{cor}  

\begin{remark} When $h \in \FF^\ast,$   the algebra  $\A_h$ is isomorphic to the Weyl algebra $\A_1$.
There are no prime polynomial  factors of $h$ in this case,  Thus, when $\chara(\FF) = 0$,  all the irreducible generalized
weight modules for $\A_1$ are induced modules $\UU(\m) = \A_1 \ot_\DD \DD/\m$ for some maximal ideal 
$\m$ of $\DD$ by Corollary \ref{cor:genwt0}.  Modules for the Weyl algebra $\A_1$, and more generally
for the Weyl algebras in arbitrarily many variables,  and for generalized Weyl algebras  over  fields of arbitrary characteristic,  have been studied
extensively by many authors (see for example, \cite{B1}, \cite{B2}, \cite{Bl81}, \cite{DGO}, \cite{C},  \cite{BBF}). \end{remark}   

\begin{subsection}{Finite-dimensional irreducible $\A_h$-modules when $\chara(\FF) = 0$} \end{subsection}

When $\chara(\FF) = 0$, Lemma \ref{lem:ann} shows that  for any finite-dimensional irreducible $\A_h$-module  $\V$,  there is a $\delta$-invariant maximal ideal $\m = \DD f$ such that $\V = \V_\m$, and $f$ is a prime
factor of $h$.    Here we determine more information about these finite-dimensional modules first in the algebraically closed case, then for
arbitrary $\FF$.  \bigskip

\begin{subsubsection}{$\FF$ algebraically closed of characteristic 0} \end{subsubsection}  

Let  $\M$ be a finite-dimensional irreducible $\A_h$-module.  As noted above,  we may assume
$\M = \M_\m$ where $\m$ is the maximal ideal generated by a prime factor $f$ of $h$, and  $x$ and $y$ are
 commuting transformations on $\M$ (compare \eqref{eq:xyw}).    When $\FF$ is algebraically closed, this implies that $x$ and $y$
 have a common eigenvector, which then is a basis for $\M$ by irreducibility.   
  Since    $f$  must be  a linear factor of $h$ in this case,  we have the following.    \medskip

\begin{thm}\label{thm:fd} Assume $\FF$ is an algebraically closed field of characteristic 0 and $h \not \in \FF$.   Then every finite-dimensional irreducible $\A_h$-module $\M$ is  one dimensional.    In particular,   there exist  $\lambda,\mu \in \FF$, with
$\lambda$  a root of $h$,   so that   $\M \cong \V_{\lambda, \mu} := \FF v_{\lambda, \mu}$, where the $\A_h$-module action is given by $x.v_{\lambda,\mu} = \lambda v_{\lambda,\mu}$ and $y.v_{\lambda, \mu} = \mu v_{\lambda,\mu}.$
Thus, in the notation of Theorem \ref{T:induce},   $\M \cong \V_{\lambda,\mu} \cong \Lmg$, where
$f = x-\lambda$,  $\m = \DD f$, and $g = y-\mu$.  
\end{thm}   
 
\begin{remark}  For the algebra $\A_x$, which is the universal enveloping algebra of the $2$-dimensional
solvable, non-abelian Lie algebra,  Theorem  \ref{thm:fd} is Lie's theorem.  For the algebra
$\A_{x^2}$, this result appears in \cite{I12}.   In both these cases (and more generally when  $h = x^n$ for any $n \geq 1$)  $\lambda = 0$ in Theorem \ref{thm:fd}.     \end{remark} 
 
\begin{cor}\label{cor:one-dim}  Assume $\FF$ is an algebraically closed field of characteristic 0,  and let $\V  = \V^\m$ be
an irreducible generalized weight module for $\A_h$ with $\m = \DD f$.   Either  
\begin{itemize}  
\item[{\rm (i)}] $f  = x - \lambda$, where $\lambda$ is a root of $h$,  and $\V = \V_{\lambda, \mu}$ for some $\mu \in \FF$, where $\V_{\lambda, \mu}$ is the one-dimensional $\A_h$-module 
determined by $\lambda,\mu$  in Theorem \ref{thm:fd};   or  
\item[{\rm (ii)}]  $f$ is not a factor of $h$ and $\V$ is isomorphic to the induced module $\UU(\m) =
\A_h \ot_{\DD} \DD/\m$.   
\end{itemize}  \end{cor}

\begin{subsubsection}{$\FF$ an arbitrary field of characteristic 0}  \end{subsubsection}

 Assume $\FF$ is an arbitrary field of characteristic 0,  and $\M$ is as above,  a finite-dimensional 
 irreducible $\A_h$-module. We may suppose  that $\M = \M_\m$, where $\m$ is a maximal ideal generated by a prime factor $f$  of $h$ of degree $d$.  By  Corollary \ref{cor:genwt0}, we know that $\M \cong  \Lmg = \UU(\m)/\bar g \UU(\m)$ for some monic prime
 polynomial $\bar g = y^n - \sum_{j=0}^{n-1} y^j \bar g_j \in \PR_\m = (\DD/\m)[y]$.   Taking $v$ any  nonzero element of $\M$,  we have that  $\{y^k x^\ell v  \mid 0 \leq k < n,  0 \leq \ell < d\}$
is a basis for  $\M$. 

Assuming  $f = x^d - \sum_{i=0}^{d-1} \zeta_i x^i$ and $g = y^n - \sum_{j=0}^{n-1} y^jg_j$,   where $\zeta_i \in \FF$ for all $i$  and the polynomial $g_j \in \DD$ is of degree less than $d$ for all $j$,  we have   

\begin{eqnarray*} x. y^k x^\ell v &=&  \begin{cases}   y^k x^{\ell+1} v   & \qquad \hbox{if} \ \  0 \leq \ell < d-1, \\
\displaystyle{\sum_{i=0}^{d-1} \zeta_i\, y^k x^i  v}    & \qquad \hbox{if} \ \  \ell =  d-1,   \end{cases} \\
  y. y^k x^\ell v &=&  \begin{cases}   y^{k+1}x^j v   & \qquad \hbox{if} \ \  0 \leq k < n-1, \\
\displaystyle{\sum_{j=0}^{n-1} y^{j} g_j x^\ell}  v = 
\displaystyle{\sum_{j=0}^{n-1} y^{j} s_{j,\ell}}  v & \qquad \hbox{if} \ \  k =  n-1, \end{cases}
\end{eqnarray*} 
where $s_{j,\ell}$ is the remainder when $g_j x^\ell$ is divided by $f$.  
   \bigskip 
 
\begin{exam}   Assume  $h = (x-\lambda)^\ell$ for some $\lambda \in \FF$ and some $\ell \geq 1$;  
 $f = x-\lambda$;  and $\m = \DD f$.  Let $g = y^n - \sum_{j=0}^{n-1}y^jg_j \in \A_h$ be such that  $g_j \in \DD$ for all $j$ and $\bar g =  y^n - \sum_{j=0}^{n-1}y^j\bar g_j$ is prime in    $(\DD/\m)[y]$, i.e.  $y^n - \sum_{j=0}^{n-1}g_j(\lambda) y^j$ is a prime polynomial in $\FF[y]$.   Then the irreducible module $\Lmg = \UU(\m)/\bar g \UU(\m)$ has a  basis $\{y^kv \mid 0\leq k <  n\}$,  where $v := u_\m + \bar g \UU(\m)$, and  the  $\A_h$-action is given by   $$x.y^kv = \lambda y^kv, \quad   y.y^kv = y^{k+1}v \ \ (0 \leq k < n-1), \quad y.y^{n-1}v = \sum_{j=0}^{n-1} g_j(\lambda) y^j v.$$   \end{exam}

\begin{subsection}{Irreducible $\DD$-torsion-free $\A_h$-modules when $\chara(\FF) = 0$}  \end{subsection}

In order to discuss  the $\DD$-torsion-free irreducible $\A_{h}$-modules when $\chara(\FF) =0$,  we assume $\Sf = \DD \setminus \{0\}$ and   $\EE = \Sf^{-1} \DD$ is the field of fractions of $\DD = \FF[x]$ as in Section 2.  The localization 
$\BB = \Sf^{-1}\A_{h}$ is the Ore extension $\BB = \EE[y,\mathsf{id}_{\EE},\delta]$,  where $\delta(e)=e'h$ for all $e \in \EE$. (Note that $\BB$ does not depend on $h$, up to isomorphism.) First we briefly review Block's correspondence between $\widehat \A_{h}(\DD\hbox{-}\mathsf{torsion\hbox{-}free})$ and $\widehat \BB(\A_{h}\hbox{-}\mathsf{socle})$, where the latter denotes the set of isomorphism classes of irreducible $\BB$-modules $\V$ such that $\mathsf{Soc}_{\A_{h}}(\V)\neq 0$. Recall that the \textit{socle} of an $\A_h$-module $\V$ is the submodule $\mathsf{Soc}_{\A_{h}}(\V)$  generated by the irreducible $\A_h$-submodules of $\V$. Block's  correspondence   
\cite[Lem.\ 2.2.1]{Bl81} gives the following  (see also \cite[Sec.\ 5]{Bv99} for the same correspondence in a more general setting).  \medskip

\begin{prop}\label{P:tfcorres} 
Let $\M$ be an irreducible $\DD$-torsion-free $\A_{h}$-module. Then $\Sf^{-1}\M = \BB \otimes_{\A_{h}} \M$ is an irreducible $\BB$-module, and the map 
\begin{equation}\label{E:irrtorfreecorr}
\widehat \A_{h}(\DD\hbox{-}\mathsf{torsion\hbox{-}free}) \longrightarrow\widehat \BB(\A_{h}\hbox{-}\mathsf{socle}), \qquad [\M]\mapsto [\Sf^{-1}\M]
\end{equation} is a bijection.
\end{prop}

\begin{proof} Let $\M$ be an irreducible $\DD$-torsion-free $\A_{h}$-module. Then $\Sf^{-1}\M = \BB \otimes_{\A_{h}} \M$ is an irreducible $\BB$-module. Thus,  there is a map
$\Psi: \widehat \A_{h}(\DD\hbox{-}\mathsf{torsion\hbox{-}free}) \longrightarrow\widehat \BB$
given by $[\M]\mapsto [\Sf^{-1}\M].$  Since $\M$ embeds in $\Sf^{-1}\M$ as an $\A_{h}$-module, we have $\M\subseteq \mathsf{Soc}_{\A_{h}}(\Sf^{-1}\M)$.  

Recall that a submodule of a module $\V$ is said to be  \textit{essential} if its intersection with any nonzero submodule of $\V$ is nonzero. It is easy to see that $\M$ is an essential $\A_{h}$-submodule of $\Sf^{-1}\M$, thus $\mathsf{Soc}_{\A_{h}}(\Sf^{-1}\M)=\M$.   This shows that the map $\Psi$ is injective,  and its image is contained in $\widehat \BB(\A_{h}\hbox{-}\mathsf{socle})$. Conversely, if $\V$ is an irreducible $\BB$-module such that $\mathsf{Soc}_{\A_{h}}(\V)\neq 0$, then we claim that $\mathsf{Soc}_{\A_{h}}(\V)$ is an irreducible $\A_{h}$-module and $\Sf^{-1}\mathsf{Soc}_{\A_{h}}(\V)=\V$.  Indeed,  if $\mathsf{L}\subseteq \mathsf{Soc}_{\A_{h}}(\V)$ is an irreducible $\A_{h}$-submodule,  then $\Sf^{-1}\mathsf{L}=\V$ by the irreducibility of $\V$ and the fact that $\mathsf{L}\subseteq\V$ is $\DD$-torsion-free.  Thus,  $\mathsf{L}$ is an essential $\A_h$-submodule of $\V$ which implies that $\mathsf{Soc}_{\A_{h}}(\V)=\mathsf{L}$ is irreducible. Hence, $\Psi$ gives  a bijection onto $\widehat \BB(\A_{h}\hbox{-}\mathsf{socle})$, with inverse
\begin{equation}
\widehat \BB(\A_{h}\hbox{-}\mathsf{socle}) \longrightarrow \widehat \A_{h}(\DD\hbox{-}\mathsf{torsion\hbox{-}free}), 
\quad [\V]\mapsto [\mathsf{Soc}_{\A_{h}}(\V)].
\end{equation} \end{proof}

Since $\BB$ is an Ore extension over the field $\Sf^{-1} \DD$, $\BB$ is a principal left ideal domain so that the irreducible $\BB$-modules are the $\BB$-modules of the form $\BB/\BB b$, where $b\in\BB$ is an irreducible element. In particular, any $\DD$-torsion-free irreducible $\A_{h}$-module has the form $\A_{h}/(\A_{h}\cap\BB b)$, for $b\in\BB$ irreducible, but not all such $\A_{h}$-modules are irreducible (compare \cite[Thm.\ 4.3]{Bl81}). In \cite[Cor.\ 2.2, Cor.\ 4.4.1]{Bl81},  Block showed that for the Weyl algebra $\A_{1}$,  the  map  $\Psi: \widehat \A_{1}(\DD\hbox{-}\mathsf{torsion\hbox{-}free}) \longrightarrow \widehat \BB$ is in fact surjective (i.e., $\widehat \BB=\widehat \BB(\A_{1}\hbox{-}\mathsf{socle})$), so the irreducible $\DD$-torsion-free $\A_{1}$-modules  correspond to $\BB$-modules of the form $\BB/\BB b$ and  are classified by the similarity classes of irreducible elements of $\BB$. This does not hold for $\A_{h}$ if $h\notin\FF$, by~\cite[Cor.\ 4.4.1]{Bl81}.  We illustrate this phenomenon with a specific example. 

\begin{exam}   Suppose $\chara(\FF)=0$.
Let  $\BB=\Sf^{-1}\A_{h}$, and consider the $\BB$-module $\BB/\BB y$. Then as an $\Sf^{-1}\DD$-module, $\BB/\BB y\cong \FF(x)$, the field of fractions of $\DD$, with $y.\frac{q}{r}=h\left(\frac{q}{r}\right)'=h\frac{q'r-qr'}{r^{2}}$ for all $q,r \in \DD$, $r \neq 0$. It is clear that $\BB/\BB y$ is an irreducible $\BB$-module, as $h^{-1}y$ acts as $\frac{d}{dx}$. Now consider the $\A_{h}$-submodule  $\A_{h}/(\A_{h}\cap\BB y)=\A_{h}/\A_{h}y$. As an $\DD$-module, $\A_{h}/\A_{h}y\cong \FF[x] = \DD$, with $y$ acting as $h\frac{d}{dx}$. For any $k\geq 0$, $h^{k}\DD$ is an $\A_{h}$-submodule of $\A_{h}/\A_{h}y$ and $\left\{ h^{k}\DD \right\}_{k\geq 0}$ 
is a strictly descending chain of submodules of $\A_{h}/\A_{h}y$ if $h\notin\FF$. In particular, $\A_{h}/\A_{h}y$ is  irreducible if and only if $h\in\FF^{*}$.   

Similarly, suppose $\mathsf{Soc}_{\A_{h}}(\BB/\BB y)\neq 0$, and assume  $\A_{h}.\frac{q}{r}\subseteq \mathsf{Soc}_{\A_{h}}(\BB/\BB y)$ is an irreducible $\A_{h}$-submodule of $\BB/\BB y \cong \FF(x)$. As $0\neq q=r\frac{q}{r}\in \A_{h}.\frac{q}{r}$,  which is an irreducible submodule, we have $\A_{h}.q=\A_{h}.\frac{q}{r}$, so we can assume $r=1$; in particular, $\A_{h}.q\subseteq \DD$. The irreducibility argument also shows that $\A_{h}.(hq)=\A_{h}.q$, so $q\in\A_{h}.(hq)$. Assume further that $h\notin\FF$ and take $k\geq 0$ maximal such that $h^{k}$ divides $q$. Then every nonzero element in $\A_{h}.(hq)$ is divisible by $h^{k+1}$, which contradicts the maximality of $k$. Thus, $\mathsf{Soc}_{\A_{h}}(\BB/\BB y)= 0$ if $h\notin\FF$. If $h\in\FF^{*}$,  then clearly $\mathsf{Soc}_{\A_{h}}(\BB/\BB y)=\A_{h}.1=\DD$.
\end{exam} 
 
 Next we will characterize the isomorphism classes of irreducible $\DD$-torsion-free $\A_{h}$-modules in terms of the irreducible $\DD$-torsion-free $\A_{1}$-modules, without involving localization. For this, we will view $\A_{h}$ as a subalgebra of the Weyl algebra $\A_{1}$ via the embedding $\A_{h}\hookrightarrow\A_{1}$, $x\mapsto x$, $\hat y\mapsto yh$, where $x, \hat y$ are the generators of $\A_{h}$ with $[\hat y, x]=h$ and $x, y$ are the generators of the Weyl algebra, satisfying $[y, x]=1$.

Let $\M$ be an irreducible $\DD$-torsion-free $\A_{h}$-module. Since $h$ is normal in $\A_{h}$, (that
is $h\A_h = \A_h h$,  as shown in  \cite[Lem.~7.1]{BLO1}),   $h\M$ is a submodule. 
But then  $h\M=\M$,  as $\M$ is $\DD$-torsion-free.   Given $m\in\M$,  there exists an $\widetilde m \in\M$ with $m=h\widetilde m$, and $\widetilde m$ is unique  since $\M$ is $\DD$-torsion-free.  Define 
$$
y.m :=\hat y.\widetilde m.
$$
It is apparent that this extends the action of $\A_{h}$ on $\M$ to an action of $\A_{1}$ on $\M$, so that $\M$ is an irreducible $\DD$-torsion-free $\A_{1}$-module.  Thus,  we have an injective map
\begin{equation}\label{E:ah2a1}
\widehat \A_{h}(\DD\hbox{-}\mathsf{torsion\hbox{-}free}) \longrightarrow \widehat \A_{1}(\DD\hbox{-}\mathsf{torsion\hbox{-}free}), 
\quad [\M]\mapsto [\M].
\end{equation}
The next result describes the image of this map.

\begin{prop}\label{P:imageah2a1}
Suppose $\M$ is an irreducible $\DD$-torsion-free $\A_{1}$-module. The following conditions are equivalent:
\begin{itemize} 
\item[{\rm (i)}]   The restriction of $\M$ to $\A_{h}$ is an irreducible $\A_{h}$-module.
\item[{\rm (ii)}] $\mathsf{Soc}_{\A_{h}}(\M) \neq 0$.
\item[{\rm (iii)}]  $h\M=\M$ and $\M$ is a Noetherian $\A_{h}$-module.
\end{itemize}
\end{prop}

\begin{proof}
The implication (i)$\implies$(ii) is obvious,  and  (i)$\implies$(iii) follows from the preceding considerations. Suppose $\mathsf{Soc}_{\A_{h}}(\M) \neq 0$, and let $\mathsf{L}$ be an irreducible $\A_{h}$-submodule of $\M$. Then as before, $h\mathsf{L}=\mathsf{L}$,  and $\mathsf{L}$ is an $\A_{1}$-submodule of $\M$. Thus $\mathsf{L}=\M$ which shows that $\M$ is an irreducible $\A_{h}$-module,  so that  (ii)$\implies$(i)  holds.

Finally, assume that $h\M=\M$ and $\M$ is a Noetherian $\A_{h}$-module. Let $\N$ be a maximal $\A_{h}$-submodule of $\M$. Thus, since $h$ is normal, $\{ m\in\M\mid hm\in\N \}$ is an $\A_{h}$-submodule of $\M$ containing $\N$. As $\N$ is maximal and $h\M=\M\not\subseteq\N$, it follows that $\{ m\in\M\mid hm\in\N\}=\N$. Given $v \in\N\subseteq \M=h\M$, there exists $m\in\M$ so that $v=hm$; hence $m\in\N$ and $h\N=\N$. Now we can conclude that $\N$ is a proper $\A_{1}$-submodule of $\M$. Therefore, $\N=0$, proving that $\M$ is an irreducible $\A_{h}$-module. This shows 
that  (iii)$\implies$(i).   \end{proof}

\begin{section}{Irreducible $\A_h$-modules when $\chara(\FF) = p > 0$}\end{section}


In this section, we investigate the irreducible $\A_h$-modules when $\chara(\FF) = p > 0$ and completely determine them when $\FF$ is algebraically closed.  When $\chara(\FF) = p>0$, all irreducible $\A_{h}$-modules are finite dimensional by Proposition \ref{P:findimp} and therefore have $\DD$-torsion.  We have seen in Theorem \ref{T:center}  that  the center of $\A_h$ is the  polynomial algebra  $\centh = \FF[x^p, z_p]$, where $z_p =   y  (y  + h') \cdots (y  + (p-1)h') = y^p-y\frac{\delta^p(x)}{h(x)}$, and $\frac{\delta^p(x)}{h(x)} \in  \FF[x^p]$.  Quillen's extension of Schur's Lemma tells us that $\centh$  must act as scalars on any irreducible $\A_h$-module $\V$ when $\FF$ is  algebraically closed.

\medskip
Since our ultimate goal is a description of the irreducibles when $\FF$ is
algebraically closed,  we make the following assumptions throughout the section: \medskip

\begin{asspts} \label{A:charphypos} {\it 
 $\V$ is an irreducible $\A_h$-module, and   there exist scalars $\beta \in \FF$ and  $\lambda, \alpha$ in the algebraic closure $\overline \FF$ of $\FF$ such that  $\lambda^p, \alpha^{p-1} \in \FF$,  and as transformations on $\V$,
 \begin{itemize}
 \item  $x^p = \lambda^p\, \mathsf{id}_\V$  (equivalently,   $(x-\lambda \, \mathsf{id}_\V)^p = 0$), 
 \item  $\frac{\delta^p(x)}{h(x)} = \alpha^{p-1}\, \mathsf{id}_\V$, 
 \item $y^p - \alpha^{p-1}y = \beta\,\mathsf{id}_\V$. 
 \end{itemize} } \end{asspts}

 Suppose $\mu \in \overline \FF$ is a root of the polynomial  $g(t):= t^p - \alpha^{p-1}t - \beta$.   Then 
 
 \begin{equation}\label{eq:roots}  \Theta = \{ \mu + j \alpha  \mid  j = 0,1, \dots, p-1\}  \end{equation} 
is the complete set of roots of $g(t)$  in $\overline \FF$.
Now  if $g(t)$  has a monic factor in $\FF[t]$, say of degree $m$
where $1 \leq m  < p$, then the coefficient of $t^{m-1}$ in that factor has the form $-(m  \mu + n\alpha)$ for some $n$.     This implies $\mu + m^{-1}n \alpha \in \FF$, hence $g(t)$
has a root in $\FF$.     From this we see that either  $t^p - \alpha^{p-1}t - \beta$ has
a root in $\FF$ or is a prime polynomial in $\FF[t]$.

\begin{lemma} \label{lem:lambda=alpha}
Suppose $\V$ is an $\A_h$-module and $\lambda \in \FF$  is such that $h(\lambda) = 0$ and $x = \lambda \, \mathsf{id}_\V$ as a transformation on $\V$.  Then  $\frac{\delta^p(x)}{h(x)} = h'(\lambda)^{p-1} \, \mathsf{id}_\V$.
 \end{lemma}
 \begin{proof}
Note that $\delta^1(x) = h$ and $\delta^2(x) =   h' h$.  It is evident by induction that  for all $k \geq 1$, \, $\delta^k(x) = (h')^{k-1}h + f_k h^2$ for some $f_k \in \DD$ (compare Lemma \ref{L:dkf} and Corollary \ref{C:dkx} below).  Therefore $\frac{\delta^p(x)}{h(x)} = h'(x)^{p-1} + f_p(x)h(x)$ and $\frac{\delta^p(x)}{h(x)} \Big |_{x = \lambda} = h'(\lambda)^{p-1} + 0$.
\end{proof}

\begin{thm}\label{T:Vdot}
Suppose $\chara(\FF) = p > 0$,  and  let $\V$ be an irreducible $\A_h$-module satisfying  the assumptions in 
\ref{A:charphypos}.  Suppose further  that $\lambda \in \FF$.     Then  one of the following holds:
\begin{itemize}
\item[{\rm (i)}]   $h(\lambda)= 0$ and  there exists $\theta \in \Theta \cap \FF$ so that  $\V = \FF v$   where $x.v = \lambda v, \  y.v = \theta v$. 
\item[{\rm (ii)}]   $h(\lambda) = 0$,  $\Theta \cap \FF = \emptyset$, and $\V$  has a basis
$\{v_n \mid n = 0,1, \dots, p-1\}$ such that $x.v_n = \lambda v_n$ for all $n$, \  $y.v_n = v_{n+1}$ for $n < p-1$ \
and \ $y.v_{p-1} = h'(\lambda)^{p-1}v_1 + \beta v_0$.  
\item[{\rm (iii)}]   $h(\lambda) \neq 0$ and $\V$  has a basis
$\{v_n \mid n = 0,1, \dots, p-1\}$ such that 
\begin{eqnarray*}  y.v_n &=&  \begin{cases}  v_{n+1} & \qquad \hbox{\rm if} \ \ \   0 \leq n < p-1, \\
\alpha^{p-1}v_1 + \beta v_0  & \qquad  \hbox{\rm if}  \ \ \ n = p-1;   \end{cases}  \\
x.v_n &=&  \sum_{j=0}^n  (-1)^j {n \choose j}  \delta^j(x) \mid_{x = \lambda} v_{n-j}. 
\end{eqnarray*}  
\end{itemize} 
\end{thm} 
\begin{proof}  Assume first that $h(\lambda) = 0$.    Then $x-\lambda$ is a factor of $h$ and $\m=\DD(x-\lambda)$ is a maximal $\delta$-invariant ideal. Since $\V_{\m}\neq 0$, Proposition \ref{prop:genwt}\,(iii) implies that $\V=\V_{\m}$. In particular, $x = \lambda \, \mathsf{id_\V}$,
and $x$ and $y$ commute as   transformations on $\V$.   Since $y$ satisfies the polynomial $t^p-\alpha^{p-1}t - \beta$ on $\V$, $\V$ is a homomorphic image of the module $\Lmg = \UU(\m)/ \bar g \UU(\m)$, where $\bar g(y)=y^p-\alpha^{p-1}y - \beta$, under the identification $\PR_\m = (\DD/\m)[y] \cong \FF[y]$.   By Lemma \ref{lem:lambda=alpha}, we may write $\bar g(y) = y^p - h'(\lambda)^{p-1} y - \beta$, where $h'(\lambda) \in \FF$ since $\lambda \in \FF$.  We have seen that either $\bar g$ has a root in $\FF$ or is a prime polynomial.  If $\mu \in \FF$ is a root of $\bar g$, then $\Theta = \{ \mu + j h'(\lambda) \mid j = 0, 1, \ldots, p-1 \} \subseteq \FF$ is the complete set of roots of $\bar g$, and it follows that $y$ has an eigenvalue $\theta \in \FF$ on $\V$, so case (i) holds.  If $\bar g$ is prime in $\PR_\m$, then $\Lmg$ is irreducible, so $\V=\Lmg$ and $\dim_{\FF}\V=p$, by Theorem~\ref{T:induce}\,(ii). Taking a nonzero vector $v_0 \in \V$ and setting $v_n = y^n.v_0$
for $n=0,1,\dots, p-1$,  we see that the $v_n$ are linearly independent, and hence are a basis of $\V$.    Moreover,
$y.v_n = v_{n+1}$ for $n < p-1$ and $y.v_{p-1} = y^p.v_0 = \alpha^{p-1} y.v_0 + \beta v_0 = 
\alpha^{p-1}v_1 + \beta v_0$, so we have case (ii).  

Now suppose that $h(\lambda) \neq 0$, and take $0 \neq v_0 \in \V$ such that $x.v_0 = \lambda v_0$. 
Assume  $v_m = y^m.v_0$ for $m= 0,1,\dots.$  Let $n$ be minimal such that there is a dependence relation $v_n  = \sum_{k=0}^{n-1}  \xi_k  v_k$.   Observe that $n \leq p$,  as the minimum polynomial in $\FF[t]$ of $y$ on $\V$ divides $t^p - \alpha^{p-1}t - \beta$.       Applying $x$ to this relation and using \eqref{eq:xynid}, we obtain
\begin{eqnarray}\label{eq:xvn} x.v_n &=& xy^n. v_0 = \sum_{j=0}^n (-1)^j {n \choose j} \delta^j(x)\, {\big |}_{x = \lambda} v_{n-j} \\ &=& \sum_{k=0}^{n-1} \xi_k  \sum_{\ell=0}^{k} (-1)^\ell {k \choose \ell}  \delta^{\ell}(x)\, {\big |}_{x = \lambda} v_{k-\ell}. \nonumber
\end{eqnarray}  
The $j = 0$ term cancels with the sum of the $\ell = 0$ terms on the right  by the minimal dependence relation.   The term $v_{n-1}$ occurs in the resulting expression only when  $j = 1$,  and in this case,  we have $(-1)n \delta(x)\, {\big |}_{x = \lambda} v_{n-1}$.    Since $\delta(x){\big |}_{x = \lambda} = h(\lambda) \neq 0$, we will achieve a dependence relation involving $v_{n-1}$, except when $n = p$.   Thus, we have case (iii).     \end{proof}

\begin{lemma}\label{L:Vlb}  Let $\chara(\FF) = p > 0$ and $\beta,\lambda \in \FF$,  and assume $h(\lambda) \neq 0$.   Let  $\m = \DD(x-\lambda)$
and set $z_\beta = y^p - y\frac{\delta^{p}(x)}{h(x)} - \beta$.     Then the quotient 
$\Lmgb: = \UU(\m)/ z_\beta \UU(\m)$   is a $p$-dimensional  irreducible $\A_h$-module with  basis $v_n = y^n. \overline u_\m$, $0 \leq n < p$,   where $\overline u_\m$ is the image of   $u_\m = 1 \ot (1 + \m)$ in
$\Lmgb$.  The $\A_h$-action is
given by 
\begin{eqnarray}\label{eq:Ahactp}  y.v_n &=&  \begin{cases}  v_{n+1} & \qquad \hbox{\rm if} \ \ \   0 \leq n < p-1, \\
\frac{\delta^p(x)}{h(x)}  \Big |_{x = \lambda} v_1 + \beta v_0  & \qquad  \hbox{\rm if}  \ \ \ n = p-1;   \end{cases}  \\
x.v_n &=&  \sum_{j=0}^n  (-1)^j {n \choose j}  \delta^j(x) \mid_{x = \lambda} v_{n-j}.\nonumber \end{eqnarray}
  \end{lemma}   
  
\begin{proof}   Since $y^p - y \frac{\delta^p(x)}{h(x)}$ is central in $\A_h$,   it is apparent that 
$z_\beta \UU(\m)$ is a submodule of $\UU(\m)$, and hence that  the corresponding quotient
$\Lmgb$ is an $\A_h$-module.    As  $\{y^n. u_\m \mid n = 0,1,\dots\}$ is a basis for $\UU(\m)$,   the module
$\Lmgb$ is spanned by the  vectors $y^n. \overline u_\m$,  $n = 0,1,\dots$, where
$\overline u_\m$ is the image of $u_\m$ in $\Lmgb$.   However, since 
 $y^p. \overline u_\m  =  \frac{\delta^p(x)}{h(x)}\big |_{x = \lambda}y. \overline u_\m + \beta \overline u_\m$, we see that 
the dimension of $\Lmgb$ is at most $p$.   The argument  that the vectors $v_n: = y^n.\overline u_\m$ are linearly
independent for $n = 0,1,\dots, p-1$ is the same as that given in  \eqref{eq:xvn}.  

Now if $\mathsf{W}$ is a nonzero submodule of $\Lmgb$, and $0 \neq w = \sum_{k=0}^n  \gamma_k v_k \in \mathsf{W}$ with $n$ minimal,   then $$(x-\lambda).w =  \gamma_n  \sum_{j=1}^n  (-1)^j {n \choose j}  \delta^j(x) \mid_{x = \lambda} 
v_{n-j} - \sum_{k=1}^{n-1} \gamma_k  \sum_{\ell=1}^k (-1)^\ell{k \choose \ell}  \delta^\ell(x) \mid_{x = \lambda}v_{k-\ell},$$
will give a smaller length element in $\mathsf W$  if  $(-1) {n \choose1} \gamma_n \delta(x) \mid_{x = \lambda}  = -n \gamma_n h(\lambda) \neq 0$.  As $h(\lambda) \neq 0$,  it must be that $n = 0$, and  $w = \gamma_0 v_0$. 
But then applying $y^n$ to $w$ shows that $v_n \in \mathsf{W}$ for all $n= 0,1,\dots, p-1$.   Hence, $\mathsf{W} =
\Lmgb$, which is irreducible. 
\end{proof}

\begin{cor}\label{C:algclo}  Assume $\FF$ is an algebraically closed field of characteristic $p > 0$,  and  let $\V$ be an irreducible $\A_h$-module.  Then either 
\begin{itemize}
\item[{\rm (i)}] for some $\lambda, \theta \in \FF$ with $h(\lambda) = 0$,  \    $\V \cong \V_{\lambda, \theta} = \FF v_{\lambda, \theta}$,   where 
$$x.v_{\lambda, \theta} = \lambda v_{\lambda, \theta}, \quad   y.v_{\lambda, \theta} = \theta v_{\lambda, \theta},  \quad  \hbox{\rm or}$$
\item[{\rm (ii)}]  for some $\lambda, \beta \in \FF$ with $h(\lambda) \neq 0$,  
$\V \cong  \Lmgb= \bigoplus_{n=0}^{p-1}  \FF v_n$,  where $\m = \DD(x-\lambda)$ and  the action of $\A_h$ is given in 
\eqref{eq:Ahactp}.  
\end{itemize}
Hence, if $\mathsf{P}$ is a primitive ideal of $\A_h$,  then $\mathsf{P}$ is isomorphic to
one of the following:
\begin{itemize}  
\item[{\rm (i)}]  $\ann_{\A_h}(\Lmg)$ for some $\m = \DD(x-\lambda)$, where $h(\lambda) = 0$,
and some $g = y- \theta$, where  $\theta \in \FF$, or
\item[{\rm (ii)}]  $\ann_{\A_h}(\Lmgb)$ for some $\m = \DD(x-\lambda)$, where $h(\lambda) \neq 0$,
and some $z_\beta = y^p - y\frac{\delta^{p}(x)}{h(x)} - \beta \in \centh$,  where $\beta \in \FF$. 
\end{itemize}
\end{cor}  

\begin{proof} This is a direct consequence of Theorem \ref{T:Vdot}  and Lemma \ref{L:Vlb},  since only
cases (i) and (iii) of that theorem occur when $\FF$ is algebraically closed.  In case (iii),  $\V$ must be a homomorphic image of
the irreducible $\A_h$-module $\Lmgb$ for some $\lambda$ and $\beta$ in $\FF$ by
Lemma \ref{L:Vlb},
so $\V$ must be isomorphic to $\Lmgb$.   
\end{proof}

\begin{section}{The Combinatorics of $\delta^k(x)$}  \end{section}

We have seen that many of the expressions for the action of $\A_h$ on an irreducible module involve terms
$\delta^k(x)$ for some $k \geq 1$, where $\delta$ is the derivation of $\DD$ given by
$\delta(f) = f' h$, and $'$ denotes the usual derivative.    Here,  we  first determine an expression for $\delta^k(f)$ for arbitrary $f$
and then specialize to the case $f = x$.    \medskip

Suppose $\nu$ is a partition of some integer $n$, and let  $\ell(\nu)$ denote the number of nonzero parts 
of $\nu$.    We write $\nu = (n^{\nu_n}, \dots, 2^{\nu_2}, 1^{\nu_1})$ to indicate that $\nu$
has $\nu_1$ parts equal to 1,  \, $\nu_2$ parts equal to 2, and so forth.  Thus,
$\sum_{k=1}^n k  \nu_k  = n$ and $\sum_{k=1}^n \nu_k = \ell(\nu)$.   For example, $\nu = (4,2^2,1^3)$ is 
a partition of 11,  which we write  $\nu \vdash 11$,  with $\nu_1 = 3$, $\nu_2 = 2$, $\nu_3 = 0$,  $\nu_4 = 1$,  and
$\ell(\nu) = 6$.   
 
Let $\emptyset$ denote the unique partition of 0 and set $h^{(\emptyset)} = 1$.  For $j \geq 1$,  let $h^{(j)} = (\frac{d}{dx})^j(h)$.   Then for   $\nu = (n^{\nu_n}, \dots, 2^{\nu_2}, 1^{\nu_1}) \vdash n$,  we define 
$$h^{(\nu)} :=  (h^{(1)})^{\nu_1}(h^{(2)})^{\nu_2}  \cdots (h^{(n)})^{\nu_n}.$$

\begin{lemma}\label{L:dkf}   For $k \geq 1$,    
$$\delta^{k}(f) = \sum_{n=0}^{k-1} \sum_{\nu \vdash n}  b_\nu^k  f^{(k-n)} h^{(\nu)} h^{k-\ell(\nu)},$$
where the $b_\nu^k$ are nonnegative integer coefficients.   \end{lemma} 
 
Before beginning the proof, and as the initial inductive steps, we present some examples:

\begin{eqnarray*}  \delta^1(f) &=&  f'h = f^{(1)}h,  \\
 \delta^2(f) &=& f'' h^2 +f' h' h = f^{(2)}h^2 + f^{(1)} h^{(1)} h,  \\
  \delta^3(f) &=& f''' h^3  + 3 f'' h' h^2 + f' h'' h^2 + f' (h')^2 h   \\
              &=&  f^{(3)} h^3 + 3 f^{(2)} h^{(1)} h^2 + f^{(1)} h^{(2)} h^2 + f^{(1)}( h^{(1)})^2 h.  
    \end{eqnarray*} 

\begin{proof}  We can assume the lemma is true for $k$ and prove it for $k+1$.   Suppose
there are nonnegative integers $b_\nu^k$ so that 

$$\delta^k(f) = \sum_{n=0}^{k-1} \sum_{\nu \vdash n}   b_\nu^k  f^{(k-n)}  h^{(\nu)} h^{k-\ell(\nu)}.$$  
Then 
\begin{eqnarray}\label{eq:dk+1}  \delta^{k+1}(f) &=& \sum_{n=0}^{k-1} \sum_{\nu \vdash n} 
  b_\nu^k  \Big( f^{(k-n)} h^{(\nu)} h^{k-\ell(\nu)}\Big)' h \\
&=&\sum_{n=0}^{k-1} \sum_{\nu \vdash n}   b_\nu^k   f^{(k+1-n)}  h^{(\nu)} h^{k+1-\ell(\nu)}+
\sum_{n=0}^{k-1} \sum_{\nu \vdash n}   b_\nu^k  f^{(k-n)} \Big( h^{(\nu)}\Big)' h^{k+1-\ell(\nu)} \nonumber \\ && \qquad  + 
\sum_{n=0}^{k-1} \sum_{\nu \vdash n}  \big (k - \ell(\nu)\big) b_\nu^k  f^{(k-n)} h^{(\nu)} h' h^{k-\ell(\nu)}.
\nonumber  \end{eqnarray} 

Observe  for $\nu \vdash n$ that 
$$\left( h^{(\nu)}\right)' = \sum_{j = 1}^n  \nu_j (h^{(1)})^{\nu_1} \cdots (h^{(j)})^{\nu_j-1}(h^{(j+1)})^{\nu_{j+1}+1} \cdots (h^{(n)})^{\nu_n}.$$
In the $j$th summand on the right,  a part of size $j$ has been converted to a part of size $j+1$.   Now if $\nu_j \neq 0$
for some $j$ such that  $1 \leq j \leq n$,   we set 
\begin{equation}\label{eq:nuj} \nu[j] = \begin{cases} (n^{\nu_n}, \dots, (j+1)^{\nu_{j+1}+1},  j^{\nu_j-1}, \dots, 2^{\nu_2}, 1^{\nu_1} ) & \ \ \  \hbox{\rm if} \ \  1 \leq j < n, \\
 ( (n+1)^{1}) \ \ \  \hbox{\rm if} \ \   j =n. \end{cases}
\end{equation}
Then  $\nu[j]  \vdash n+1$,  and $\ell(\nu[j]) = \sum_{i=1}^n \nu_i = \ell(\nu)$.       Hence,   

\begin{equation*} b_\nu^k  f^{(k-n)} \Big( h^{(\nu)}\Big)' h^{k+1-\ell(\nu)}
=   \sum_{j=1}^n  b_\nu^k \nu_j  f^{(k+1-(n+1))} h^{(\nu[j])} h^{k+1-\ell(\nu[j])},  \end{equation*} 
 where $h^{(\nu[j])}$ should be interpreted as $1$ if $\nu_j = 0$. 
 
Now let's consider a term  $h^{(\nu)} h' h^{k-\ell({\nu})}$ in the last sum of \eqref{eq:dk+1},  where $\nu \vdash n$.   
Then  $h^{(\nu)} h'$ corresponds to the partition  
\begin{equation}\label{eq:nuplus} \nu^+ = (n^{\nu_{n}}, \dots, 2^{\nu_2}, 1^{\nu_1+1})\vdash n+1,\end{equation}
which has one more part equal to 1  than does $\nu$.    Hence 
$k+1 - \ell(\nu^+) = k - \ell(\nu)$, and the corresponding term is
$$\big (k - \ell(\nu)\big)b_\nu^k  f^{(k-n)}  h^{(\nu)} h' h^{k-\ell(\nu)}  =  \big (k+1 - \ell(\nu^+)\big) b_\nu^k  f^{(k+1-(n+1))}  h^{(\nu^+)} h^{k+1-\ell({\nu^+})}.$$   

For  $\mu \vdash m$,  where $0 \leq m <  k+1$,   if 
$f^{(k+1-m)} h^{(\mu)} h^{k+1-\ell(\mu)} \neq 0$, it appears  in \eqref{eq:dk+1} with a nonnegative integer 
coefficient $b_{\mu}^{k+1}$, which is obtained from summing the following:
\begin{itemize}
\item [{\rm (i)}]  $b_{\mu}^k$   \ \  if   $m < k$,   
\item [{\rm (ii)}]  $\nu_j b_{\nu}^k$  \ \ if  $\nu \vdash m-1$ is a partition  such that 
$\nu[j] = \mu$,   
\item[{\rm (iii)}]  $(k-\ell(\nu))b_{\nu}^k$ \ \   if $\nu \vdash m-1$ is a partition such that $\nu^+ = \mu$.  \end{itemize}    Hence $b_{\mu}^{k+1}$ is
a nonnegative integer and 
$$\delta^{k+1}(f) = \sum_{m=0}^k  \sum_{\mu \vdash m}  b_{\mu}^{k+1}  f^{(k+1-m)} h^{(\mu)} h^{k+1-\ell(\mu)}.$$
\end{proof}
Since  $f^{(j)} = 0$ for all $j \geq 2$  when $f = x$,  Lemma \ref{L:dkf} reduces in this special case to

\begin{cor}\label{C:dkx}  For $k \geq 1$,    
$$\delta^{k}(x) = \sum_{\mu \vdash k-1}  c_\mu^k  h^{(\mu)} h^{k-\ell(\mu)},$$
where the coefficients  $c_\mu^k$ are nonnegative integer coefficients, which are obtained from 
the coefficients $c_{\nu}^{k-1}$ appearing in $\delta^{k-1}(x)$  by summing all the following terms: 
  \begin{itemize}
\item [{\rm (a)}]  $\nu_j c_{\nu}^{k-1}$  \ \ if  $\nu \vdash k-2$ is a partition  such that 
$\nu[j] = \mu$, where $\nu[j]$ is as in \eqref{eq:nuj};    
\item[{\rm (b)}]  $(k-1-\ell(\nu))c_{\nu}^{k-1}$ \ \  if $\nu \vdash k-2$ is a partition such that $\nu^+ = \mu$, where $\nu^+$ is as in \eqref{eq:nuplus}.  \end{itemize}     
   \end{cor}  \medskip

In the table below, for $k = 1, \dots, 7$ and for each partitition $\mu \vdash k-1$, we display  the coefficient $c_{\mu}^k$ as a subscript on $\mu$.   \bigskip

\noindent
$$\boldsymbol{\begin{array}{c|ccccccccc} 
\hspace{-.1truein} k & c_{\mu}^k &  &&&  &&&   \\
 \hline \\
\hspace{-.1truein} \bf{ 1} & \bf{(0)_{\color{red} 1}} &&&&&&  \\ \\
\hspace{-.1truein}\bf{ 2} & \bf{(1)_{\color{red} 1}} &&&&&&  \\
\\ 
\hspace{-.1truein}\bf{3} & \bf{(2)_{\color{red} 1}}  & \bf{(1^2)_{\color{red} 1}} &&&&& \\ 
\\ 
\hspace{-.1truein}\bf{4} & \bf{(3)_{\color{red} 1}}  & \bf{(2,1)_{\color{red} 4}} & \bf{(1^3)_{\color{red} 1}}&&&& \\ \\ 
\hspace{-.1truein}\bf {5} & \bf{(4)_{\color{red} 1}}  & \bf{(3,1)_{\color{red} 7}} &  \bf{(2^2)_{\color{red} 4}}&\bf{(2,1^2)_{\color{red} 11}} &\bf{(1^4)_{\color{red} 1}}&&
\\ \\  
\hspace{-.1truein}\bf{6} & \bf{(5)_{\color{red}1}}  & \bf{ (4,1)_{\color{red} 11}} &  \bf{ (3,2)_{\color{red} 15}} &\bf{(3,1^2)_{\color{red} 32}} &
\bf{(2^2,1)_{\color{red} 34}} &
\bf{(2,1^3)_{\color{red} 26}}&\bf{(1^5)_{\color{red} 1}}  \\ \\ 
\hspace{-.1truein} \bf{7} & \bf{(6)_{\color{red}1}}  & \bf{ (5,1)_{\color{red} 16}} &  \bf{ (4,2)_{\color{red} 26}} &\bf{(4,1^2)_{\color{red} 76}} &
\bf{(3^2)_{\color{red} 15}} &
\bf{(3,2,1)_{\color{red} 192}}&  \bf{ (3,1^3)_{\color{red} 122}}
 \\ \\
\hspace{-.1truein}\bf{7} & & & &   \bf{ (2^3)_{\color{red} 34}} & \bf{ (2^2,1^2)_{\color{red} 180}}&  \bf{ (2,1^4)_{\color{red} 57}} &  \bf{(1^6)_{\color{red} 1}}  \\
\hspace{-.1truein}{\bf cont.} &&&&&&&  \end{array}}$$

\smallskip
\begin{exams} (1)  Consider the partition  $\mu = (3,2)\vdash 5$, so here $k=6$.    Since  $\mu = \nu[2]$, for $\nu = (2^2)$,  and $\mu = \pi[1]$ for $\pi=(3,1)$, we have  $c_{\mu}^6  =  2c_{\nu}^5  + c_{\pi}^5  =  2 \cdot 4 + 7  =  15$, as displayed in the table. 

(2)  As another example,  consider the partition $\mu = (2^2,1) \vdash 5$.    Now  $\mu = \nu^+$ for $\nu = (2^2) \vdash 4$,  and $\mu =  \pi[1]$ for $\pi = (2,1^2) \vdash 4$.    Thus,  
$c_{\mu}^6  =  (5-\ell(\nu))c_{\nu}^5 +  2c_\pi^5  =  3 \cdot  4 +  2 \cdot 11 =  34$, as shown.  \end{exams} 
\medskip

The coefficients $c_{\mu}^k$ satisfy some intriguing properties.    We illustrate this with one particular example in
the next proposition. 

\begin{prop} Assume $c_{\mu}^k$ are the coefficients appearing in Corollary \ref{C:dkx}.    Then 
$$\sum_{\mu \vdash k-1} c_{\mu}^k  =  (k-1)\,!.$$  \end{prop}

\begin{proof}  We proceed by induction on $k$.    Verification for small values of $k$ can be done by adding
the subscripts in the $k$th  row of the  table.     We assume the result for $k$ and
show it for $k+1$.   To accomplish this, 
we define a new sort of ``multiplication'' that will help to reveal the proof. 

\begin{itemize}
\item {\bf Step 1}.  {\it  List the parts of a partition $\nu$ of $k-1$ with multiplicity in descending order, and add sufficiently many  0's to get  a $k$-tuple $\tilde \nu$ with weakly descending components.  Multiply the $k$-tuple $\tilde \nu$  by $c_{\nu}^k$,  then sum over $\nu \vdash k-1$. } 

To illustrate this,  consider the  line corresponding to $k=4$ in the table,  which is   $\boldsymbol{(3)_{\color{red} 1} \ \  (2,1)_{\color{red} 4}  \ \  (1^3)_{\color{red} 1}}$.  
In this step we rewrite it as 
$$\boldsymbol{(3,0,0,0) + {\color{red} 4}(2,1,0,0) + (1,1,1,0)}.$$

\item {\bf Step 2}.   {\it ``Multiply''  by $(1)$; i.e.  add 1 to each component in all possible ways
and sum the result.}
\begin{eqnarray*} && \hspace{-.5 truein} \bf{(1) \ast \Big((3,0,0,0) + {\color{red} 4}(2,1,0,0) + (1,1,1,0) \Big) = \qquad \qquad } \\
&&\qquad \quad  \bf{\ \ (4,0,0,0) + (3,1,0,0) + (3,0,1,0) + (3,0,0,1)}  \\
&&\qquad \quad   \bf{+  {\color{red} 4}(3,1,0,0) +  {\color{red} 4}(2,2,0,0) +{\color{red}  4}(2,1,1,0) + {\color{red} 4}(2,1,0,1)} \\
&&\qquad \quad  \bf{+  (2,1,1,0) +  (1,2,1,0) + (1,1,2,0) + (1,1,1,1).}  \end{eqnarray*}

\item{\bf Step 3}. {\it Collect terms that are the same after permutation of the components.}
$$\boldsymbol{(4,0,0,0) + {\color{red} 7}(3,1,0,0)+ {\color{red} 4}(2,2,0,0) +{\color{red} 11}(2,1,1,0) + (1,1,1,1).}$$
We can read off the line $k = 5$ in the table from this.   

\end{itemize}
This process is just a different way of doing what is described in  Corollary \ref{C:dkx} to determine
the coefficient  $c_{\mu}^{k+1}$.   Indeed,  adding 1 to the nonzero parts of a $k$-tuple
 takes into account the multiplicities
in (a) of that corollary,  and adding 1 to the $k-\ell(\mu)$ components that are 0 accounts for (b).  
Thus, the resulting coefficient of each $\mu \vdash k$ is $c_{\mu}^{k+1}$.  
Now suppose that $\sum_{\nu \vdash k-1}  c_{\nu}^k = (k-1)!.$
The sum of the coefficients in $(1) \ast  \sum_{\nu \vdash k-1}  c_{\nu}^{k} \tilde \nu$
is just $\sum_{\mu \vdash k} c_{\mu}^{k+1}$.       But each   $c_{\nu}^{k}$
appears $k$ times in $(1) \ast  \sum_{\nu \vdash k-1}  c_{\nu}^{k} \tilde \nu$.   Thus,

$$\sum_{\mu \vdash k} c_{\mu}^{k+1}  =  
k \sum_{\nu \vdash k-1} c_{\nu}^k =  k \cdot (k-1)! = k!.$$
\end{proof}


\noindent  \textsc{Georgia Benkart} \\
\textit{\small Department of Mathematics, University of Wisconsin-Madison, Madison, WI 53706, USA}\\
\texttt{benkart@math.wisc.edu}\\ 

\noindent \textsc{Samuel A. Lopes} \\
\textit{\small CMUP, Faculdade de Ci\^encias, Universidade do Porto, 
Rua do Campo Alegre 687\\ 
4169-007 Porto, Portugal}\quad 
\texttt{slopes@fc.up.pt}\\

\noindent \textsc{Matthew Ondrus} \\
\textit{\small Mathematics Department,
Weber State University, 
Ogden, Utah, 84408 USA}\\
\texttt{mattondrus@weber.edu}

\end{document}